\documentclass[11pt]{article}
\usepackage[a4paper,top=2cm,bottom=2cm,left=3cm,right=3cm,marginparwidth=1.8cm]{geometry}
\usepackage{graphicx} 
\usepackage{amsmath}
\usepackage{amsthm}
\usepackage{amssymb}
\usepackage{hyperref}
\usepackage{xcolor}
\usepackage{graphicx}
\usepackage{a4wide}
\usepackage[T1]{fontenc}
\usepackage{a4wide}
\usepackage{enumitem}
\usepackage{lipsum}
\usepackage{mathtools}
\usepackage{amsthm}
\usepackage{amsfonts}
\usepackage{cleveref}
\usepackage{yfonts}
\usepackage{amssymb}
\usepackage{xcolor}
\usepackage{mathrsfs}
\usepackage{amsmath}
\usepackage{cite}
\usepackage{epsfig}
\usepackage{stmaryrd}
\usepackage{comment}
\usepackage[utf8]{inputenc} 
\usepackage[T1]{fontenc}
\usepackage{tikz}
\usepackage{caption}
\usepackage{soul}
\usetikzlibrary{trees,patterns}
\usetikzlibrary{decorations.markings, calc}
\usepackage[normalem]{ulem}

\usepackage{mathdots}

\newcommand{\ninn}[0]{\noindent}

\newtheorem{theorem}{Theorem}
\newtheorem{proposition}[theorem]{Proposition}
\AddToHook{env/proposition/begin}{\crefalias{theorem}{proposition}}

\newtheorem{lemma}[theorem]{Lemma}
\AddToHook{env/lemma/begin}{\crefalias{theorem}{lemma}}
\newtheorem{conjecture}{Conjecture}
\newtheorem{problem}[conjecture]{Problem}
\theoremstyle{definition}

\crefname{lemma}{lemma}{lemmas}
\crefname{theorem}{theorem}{theorems}

\usepackage{setspace}
\title{Leaf-to-leaf paths and cycles in degree-critical graphs}

\author{Francesco Di Braccio\thanks{Department of Mathematics, London School of Economics, UK. Email: {\tt f.di-braccio@lse.ac.uk}} 
\and Kyriakos Katsamaktsis\thanks{Department of Mathematics, University College London, UK. Research supported by the Engineering and Physical Sciences Research Council [grant number EP/W523835/1]. Email: {\tt kyriakos.katsamaktsis.21@ucl.ac.uk} } 
\and Jie Ma\thanks{School of Mathematical Sciences, University of Science and Technology of China, Hefei, Anhui 230026, China, and Yau Mathematical Sciences Center, Tsinghua University, Beijing 100084, China.
Research supported by National Key Research and Development Program of China 2023YFA1010201 and National Natural Science Foundation of China grant 12125106. Email: {\tt jiema@ustc.edu.cn}}
\and Alexandru Malekshahian\thanks{Mathematical Institute, University of Oxford, UK. Email: {\tt alex.malekshahian@maths.ox.ac.uk}. Research completed while the author was affiliated with the Department of Mathematics, King's College London, UK.}
\and Ziyuan Zhao\thanks{School of Mathematical Sciences, University of Science and Technology of China, Hefei, Anhui 230026, China. Research supported by Innovation Program for Quantum Science and Technology 2021ZD0302902. Email: {\tt zyzhao2024@mail.ustc.edu.cn}}
}

\date{}

\begin{document}
\maketitle

\begin{abstract}

An $n$-vertex graph is \emph{degree 3-critical} if it has $2n - 2$ edges and no proper induced subgraph with minimum degree at least 3. In 1988, Erd\H{o}s, Faudree, Gy\'arf\'as, and Schelp asked whether one can always find cycles of all short lengths in these graphs, which was disproven by Narins, Pokrovskiy, and Szab\'o through a construction based on leaf-to-leaf paths in trees whose vertices have degree either 1 or 3. They went on to suggest several weaker conjectures about cycle lengths in degree 3-critical graphs and leaf-to-leaf path lengths in these so-called 1-3 trees. We resolve three of their questions either fully or up to a constant factor. Our main results are the following:
 \begin{itemize}
    \item every $n$-vertex degree 3-critical graph has  $\Omega(\log n)$ distinct cycle lengths;
    \item every tree with maximum degree $\Delta \ge 3$ and $\ell$ leaves has at least $\log_{\Delta-1}\, ((\Delta-2)\ell)$ distinct leaf-to-leaf path lengths;
    \item for every integer $N\geq 1$, there exist arbitrarily large 1--3 trees which have $O(N^{0.91})$ distinct leaf-to-leaf path lengths smaller than $N$, and, conversely, every 1--3 tree on at least $2^N$ vertices has $\Omega(N^{2/3})$ distinct leaf-to-leaf path lengths smaller than $N$.
    \end{itemize}
Several of our proofs rely on purely combinatorial means, while others exploit a connection to an additive problem that might be of independent interest.
\end{abstract}

\section{Introduction}
There is a long line of research in combinatorics seeking to understand what conditions guarantee that a graph contains cycles of many different lengths. 
In 1973, Bondy~\cite{bondy1971pancyclic} made the famous meta-conjecture that any non-trivial condition that guarantees Hamiltonicity is enough to ensure that the $n$-vertex graph is \emph{pancyclic}, i.e. that it contains all cycle lengths in $\{3, \dots, n\}$. This led to a host of interesting results in the following fifty years bringing support to Bondy's conjecture in a variety of different settings~\cite{letzter2023pancyclicity,draganic2024pancyclicity,bondy1971pancyclic,bauer1990hamiltonian, Bucić_Gishboliner_Sudakov_2022}.
However,
most of the results in the area concern (somewhat) dense graphs, and for very sparse graphs our understanding of which graphs contain many cycle lengths is more  fragmentary.
Sudakov and Verstra\"ete~\cite{sudakov2008cycle} showed that graphs with average degree \(d\) and girth at least \(g\) contain $\Omega(d^{\lfloor (g-1)/2 \rfloor})$ distinct cycles lengths, thus proving a conjecture of Erd\H{o}s~\cite{erdos1993some}.
A related conjecture of Erd\H{o}s and Hajnal~\cite{erdos1993some} was resolved by Gy\'arf\'as, Koml\'os, and Szemer\'edi~\cite{gyarfas1984distribution}, who proved that in a graph with average degree $d$, the sum of the reciprocals of the distinct cycle lengths is $\Omega(\log d)$.

The starting point of the present work is a conjecture of Erd\H{o}s, Faudree, Gy\'arf\'as, and Schelp~\cite{erdos1988cycles}, who asked whether many cycle lengths can be found in a specific class of sparse graphs called \emph{degree 3-critical graphs}. These are defined to be graphs with \(n\) vertices, \(2n-2\) edges and no proper induced subgraph with minimum degree at least 3; it is not hard to see that these graphs necessarily have minimum degree 3. Degree 3-critical graphs satisfy several interesting properties; for example, they have no proper induced subgraph $H$ on $2|V(H)|-2$ edges, and hence, by a theorem of Nash-Williams \cite{nashwill}, they are the union of two edge-disjoint spanning trees. 

Erd\H{o}s, Faudree, Gy\'arf\'as, and Schelp~\cite{erdos1988cycles} proved that any $n$-vertex degree 3-critical graph contains a cycle of length 3, 4, and 5, as well as a cycle of length at least $\log n$.\footnote{Unless indicated otherwise, logarithms throughout this paper are base 2.}
This last bound was later improved by Bollob\'as and Brightwell \cite{bollobas1989long} to $4\log n+O(\log \log n)$, which is asymptotically best possible.
In an effort to reveal a rich structure of cycle lengths in such graphs, Erd\H{o}s et al. \cite{erdos1988cycles} (also see \cite{Erd91}) conjectured that it should be possible to find cycle lengths $3,4,5 \dots, N$ for some $N= N(n) \to \infty$ as $n \to \infty$.
Their conjecture, however, was disproven by Narins, Pokrovskiy, and Szab\'o \cite{NPSz} who showed that there are arbitrarily large degree 3-critical graphs with no cycle of length 23. 
The crucial ingredient of their construction is a particular class of trees called \emph{1--3 trees}. A 1--3 tree is a tree where every vertex has degree either 1 or 3. It was shown in \cite{NPSz} that there exist infinitely many 1--3 trees with no two leaves at distance $20$ from one another, which then yielded the desired degree 3-critical graphs by adding two vertices adjacent to all leaves and to each other. 

Despite their surprising counterexamples, the authors of \cite{NPSz} proved that any degree 3-critical graph with at least six vertices contains a cycle of length 6, and asked whether it might still be the case that degree 3-critical graphs contain many cycle lengths. They posed the following conjecture.

\begin{conjecture}[\!\!{\cite[Conjecture 6.2]{NPSz}}]\label{conj:manycycles}
    Every degree 3-critical graph on $n$ vertices contains cycles of at least $3 \log n + O(1)$ distinct lengths.
\end{conjecture}

A classical construction of Bollob\'as and Brightwell \cite{bollobas1989long} shows that, if true, \Cref{conj:manycycles} is best possible. Our first result proves that \Cref{conj:manycycles} is true up to a constant factor. 

\begin{theorem}\label{thm:d-3-cCL}
     Every degree 3-critical graph on $n$ vertices contains cycles of at least $\frac{\log n}{3+\log 3}+O(1)$ distinct lengths.
\end{theorem}
 
This provides the first bound on the number of cycle lengths as a function of $n$ tending to infinity,  
and arguably can be viewed as confirmation of the original motivation of Erd\H{o}s et al.~\cite{erdos1988cycles} to demonstrate the abundance of cycle lengths in such graphs.
In fact, we establish this result as a corollary of a more general theorem (Theorem~\ref{thm:ord->uppc(G)})
which applies to \emph{degree $k$-critical graphs} for any $k \geq 3$, i.e., $n$-vertex graphs with $(k-1)n - \binom{k}{2} + 1$ edges and no proper induced subgraph with minimum degree at least $k$. This family was introduced by Bollobás and Brightwell \cite{bollobas1989long} as a natural generalization of degree 3-critical graphs, 
and a problem closely related to this family was studied more recently by Sauermann \cite{sauermann}.

The key idea behind the proof of \Cref{thm:d-3-cCL} is to define an appropriate partial ordering on the vertex set of the given graph. By Dilworth's theorem, this either gives a long chain or a long antichain. The first case yields a long path $P$ together with a collection of paths that intersect $P$ in a special way (a structure known as a \emph{vine}). In the second case, we find two large trees that are vertex-disjoint except for the fact that they share the same set of leaves. A careful analysis then yields many cycle lengths in either case. 

Motivated by the connection between degree 3-critical graphs and 1--3 trees that they established, the authors of \cite{NPSz} also formulated two conjectures about leaf-to-leaf path lengths in 1--3 trees. The first of these conjectures is as follows.

\begin{conjecture}[\!\!{\cite[Conjecture 6.3]{NPSz}}, corrected version]\label{conj:anylengths} 
    Every 1--3 tree $T$ of order $n$ has leaf-to-leaf paths of at least $\log (n+2) -1$ distinct lengths. 
\end{conjecture}
Here and throughout the rest of this paper, the length of a path is equal to the number of edges of the path, and we consider a single vertex to be a path of length 0.

The original form of \Cref{conj:anylengths} in \cite{NPSz} asks for at least $\log n$ distinct lengths, but as stated this is false, as the following example shows. For any $d \geq 2$, consider the (unique) 1-3 tree $T$ in which, for some root $r \in V(T)$, every leaf is at distance precisely $d$ from $r$. It is not hard to see that $T$ contains $3 \cdot 2^{d} -2$ vertices but only $d+1 < \log(3 \cdot 2^{d} -2)$ distinct leaf-to-leaf path lengths (namely, the ones in $\{0,2,4,\dots, 2d\}$). This example also shows that \Cref{conj:anylengths} is tight whenever $n=3\cdot 2^{d}-2$ for some $d\geq 2$.

Our second result resolves~\Cref{conj:anylengths} in a strong form. Our proof works for arbitrary trees, and gives a bound depending on the maximum degree. Consider, however, for any $n>\Delta \geq 2$, the tree obtained from a star $S_{\Delta}$ by subdividing an edge $n - \Delta - 1$ times. This yields a tree with $n$ vertices and maximum degree $\Delta$ with only three distinct leaf-to-leaf path lengths, so we cannot expect to give a bound in terms of just $n$ and $\Delta$. Instead, we require control over the number of \emph{leaves}, say $\ell$, of the tree.

\begin{theorem}\label{anylengths}
    Let $T$ be a tree with maximum degree $\Delta \ge 3$ and $\ell$ leaves. 
    Then $T$ has at least $\log_{\Delta-1}\, ((\Delta-2)\ell)$ distinct leaf-to-leaf path lengths. 
\end{theorem}
\Cref{anylengths} for \(\Delta = 3\) implies \Cref{conj:anylengths} since any 1--3 tree on \(n\) vertices has precisely \(\frac{n+2}{2}\) leaves. More generally, our result is tight whenever $\ell = \Delta (\Delta - 1)^{d-1}$ for some $d \geq 2$, as demonstrated by the tree $T$ in which each vertex has degree 1 or \(\Delta\) and each leaf is at distance precisely $d$ from some root $r \in V(T)$ (whose leaf-to-leaf path lengths are $0, 2, \dots, 2d$). In fact, noticing that $T$'s leaves can be grouped into $(\Delta-1)$-tuples of sister leaves that share a neighbour, and that deleting at most $(\Delta-2)$ leaves in each tuple does not affect the path lengths of the tree, we may construct for each $\ell'>\Delta(\Delta-1)^{d-2}$ a tree $T'$ with maximum degree $\Delta$ and $\ell'$ leaves and only $d+1$ distinct leaf-to-leaf path lengths. This shows that \Cref{anylengths} is tight for all values of $\ell$ and $\Delta$, up to an additive term of 1. The proof proceeds by finding a suitable choice of root vertex through Helly's theorem for trees, deleting the leaves that are at a certain distance from the root, and then applying induction.

While \Cref{conj:anylengths} imposes no restrictions on the lengths considered, the final conjecture of Narins, Pokrovskiy and Szab\'o \cite{NPSz} that we address asks to determine how many \emph{short} leaf-to-leaf path lengths can be found. They conjectured that for 1--3 trees, one can find path lengths which are dense in an interval of the form $[0,N]$.

\begin{conjecture}[\!\!{\cite[Conjecture 6.4]{NPSz}}]\label{conj:smalllengths}
    There exist a constant $\alpha >0$ and a function $N = N(n)$ tending to infinity as \(n \rightarrow \infty\) such that every 1--3 tree of order $n$ contains at least $\alpha N$ distinct leaf-to-leaf path lengths between 0 and $N$.
\end{conjecture}

Our next result disproves \Cref{conj:smalllengths} in the following strong form, namely, with a poly-sublinear upper bound. 

\begin{theorem}\label{thm:upper-bound-lengths}
    There exists an absolute constant $c\in (0,1)$ such that the following holds.
    
    For all $N \geq 1$ and all even $n\geq N$, there exists an $n$-vertex 1--3 tree with $O(N^c)$ distinct leaf-to-leaf path lengths between 0 and $N$.
\end{theorem}

The proof of \Cref{thm:upper-bound-lengths} yields $c=\left(2-\frac{\log 10}{\log 13} \right)^{-1}\approx0.9073$. We complement this result by also providing a polynomial lower bound on the number of short lengths that may be found, which shows that we cannot take $c < 2/3$ in \Cref{thm:upper-bound-lengths}. 

\begin{theorem}\label{smalllengths} For all $N \geq 1$ and all even $n \geq 2^{N/2}$, every $n$-vertex 1-3 tree contains leaf-to-leaf paths of $\Omega(N^{2/3})$ distinct lengths between $0$ and $N$.
\end{theorem}

In fact, \Cref{smalllengths} is an immediate corollary of a more general statement about trees with no vertices of degree 2. Given a tree $T$ and a leaf $v \in V(T)$, we say that $v$ \emph{witnesses} the length $\ell$ if there is a leaf-to-leaf path of length $\ell$ containing $v$ (as an endpoint).

\begin{theorem}\label{lots}
For all $N \geq 1$ sufficiently large, both of the following statements hold. 
\begin{enumerate}[label=(\roman*)]
    \item Let $T$ be a tree containing no vertex of degree 2. If $T$ contains a path of length at least $N/2$, then $T$ contains $\Omega(N^{2/3})$ leaf-to-leaf paths of
    distinct lengths between $0$ and $N$, all witnessed by the same leaf $v \in V(T)$. 
    \item\label{lots-upper-bound}
    For all even $n$, there exists an $n$-vertex 1--3 tree in which no leaf witnesses more than $O(N^{2/3})$ distinct lengths between 0 and $N$. 
    \end{enumerate}
\end{theorem}

Note that the assumption that there are no vertices of degree 2 in the lower bound of \Cref{lots} is necessary, as shown again by the example of a subdivided star. Since every $n$-vertex 1-3 tree has diameter at least $\log n - 2$ (for instance, by \Cref{anylengths}), we see that indeed the first part of \Cref{lots} implies \Cref{smalllengths}.

The proof of the first part of \Cref{lots} proceeds as follows: if $T$ contains many disjoint (rooted) subtrees in which some leaf is very close to the root, then we use the Erd\H{o}s-Szekeres theorem to find a subfamily of such subtrees for which we can control the lengths of paths between leaves in distinct subtrees. If instead $T$ contains a subtree $T'$ in which every leaf is far from the root, then we find many distinct leaf-to-leaf path lengths inside of $T'$.

The proofs of \Cref{thm:upper-bound-lengths} and the second part of \Cref{lots} rely on a connection to an additive combinatorics question which may be interesting in its own right. More specifically, we construct a tree $T$ by appending balanced binary trees of varying depths to a long path; it then turns out that the set of leaf-to-leaf path lengths in $T$ can be controlled by the additive structure of the sequence of subtree depths. For \Cref{thm:upper-bound-lengths}, this allows us to relate the problem to the construction of a pair of finite sets $U, V \subseteq \mathbb{N}$ such that, for some large $m \geq 1$, $U - V = [m]$ and $|U+V| = O(m^\beta)$ for some suitable $\beta\in (0,1)$. We discuss this in more detail in the concluding remarks (\Cref{sec:conclud}).

\subsection{Notation}

We use standard asymptotic notation and graph theory notation and terminology -- see \cite{bollobas}. 

In particular, given a (simple, undirected) graph $G$ we write $N_G(v)$ for the neighbourhood of a vertex $v$ in $G$, $\deg_G(v)$ for the degree of $v$ and $d_G(u, v)$ for the distance between $u$ and $v$ in a graph, i.e.\ the number of edges of the shortest path connecting them. We will drop the subscript $G$ from the above notations if the graph $G$ is clear from context.
We also write $\Delta(G)$ for the maximum degree of $G$. 
For $U\subseteq V(G)$, let $G[U]$ be the induced subgraph of $G$ with the vertex set $U$. For a path $P$ and a cycle $C$ in $G$, we denote the length of $P$ (resp. $C$) by $\ell(P)$ (resp. $\ell(C)$), meaning the number of edges in $P$ (resp. $C$).

For positive integers $s, t$, we write $(t)_s$ for the residue of $t \mod s$ (as an integer in $\{0, 1, \dots, s-1\}$) and also use the nonstandard notation $(t)_s^*$ for the same residue considered as an integer in $\{1, 2, \dots, s\}$. When $s\leq t$, define $[s,t]=\{i\in \mathbb{Z}:s\leq i\leq t\}$ and let $[t]=[1,t]$. 

Given a rooted tree $(T,r)$, its \emph{layers} are the sets $\{v \in V(T): d(v,r)=i\}$ for $i \geq 0$. Given $\ell \geq 1$, we call $(T, r)$ a \emph{perfect binary tree on $\ell$ layers} if $T$ is a binary tree rooted at $r$ and every leaf $v\in T$ satisfies $d(r, v)=\ell-1$. We denote the set of leaves of \(T\) by \(L(T)\).
For $u,v\in V(T)$, we write $T[u,v]$ to denote the unique $(u,v)$ path in $T$.

We also employ a common abuse of notation by omitting floor and ceiling symbols and ignoring the rounding errors this causes whenever it is not essential for our argument; we emphasize this will only occur in the proofs of our asymptotic results and not in the case of \Cref{anylengths}.

\subsection{Organization}
The remainder of the paper is organized as follows. We prove that we can find many leaf-to-leaf path lengths in trees -- \Cref{anylengths} and the first part of \Cref{lots} -- in \Cref{manypaths}. We provide constructions of trees with a small number of distinct leaf-to-leaf path lengths -- \Cref{thm:upper-bound-lengths} and the second part of \Cref{lots} -- in \Cref{constructions}. We prove \Cref{thm:d-3-cCL} -- that we can find many distinct cycle lengths in degree 3-critical graphs -- in \Cref{sect:cycles}. We discuss several open problems in \Cref{sec:conclud}.

\section{Finding many leaf-to-leaf path lengths}\label{manypaths}
\subsection{Paths of unrestricted length}\label{pfthm1}

In this section, we prove \Cref{anylengths}. We begin with a lemma showing how to find many lengths in a rooted tree with many leaves at the same distance from the root.

 \begin{lemma}\label{lem:same_depth}
    Let $\Delta\geq 3$ and let $T$ be a rooted tree with root $r$ and $\Delta(T) \leq \Delta$. Assume that for some $a\geq 1$ there are $m$ distinct leaves $x_1, \dots, x_m$ such that $d(r, x_i) =a$ for all $1\leq i\leq m$. 
    Then there exists an \(i \in [m]\)
    such that $T$ contains leaf-to-leaf paths of at least $\log_{\Delta - 1} (m/\Delta) + 2$ distinct lengths between $0$ and $2a$, all witnessed by $x_i$. 
\end{lemma}
\begin{proof}
Denote the root's neighbours by $r_1, \dots, r_k$ with $k\leq\Delta$. Deleting the root $r$ from $T$ gives $k$ new rooted trees $T_1, \dots T_k$, with the new roots being the $r_i$'s.

\medskip

\ninn \textbf{Case 1:} $\deg(r)\leq\Delta-1$. In this case, we will prove the slightly stronger result that we can find at least $\log_{\Delta-1}m+1$ suitable lengths, all witnessed by the same $x_i$. We proceed by induction on the number of vertices of $T$.

As a base case, note that if $T$ has only one vertex $x_1$, then there is precisely $\log_{\Delta-1}(1)+1=1$ leaf-to-leaf path, namely that of length 0 (witnessed by $x_1$).

For the inductive step, we distinguish two further subcases. If one of the $T_i$'s contains all leaves $x_1, \dots, x_m$, then the claim follows by the induction hypothesis applied to $T_i$, since the root of $T_i$ has degree at most $\Delta-1$. Otherwise, by relabelling if necessary, we may assume that $T_1$ contains at least $m/(\Delta-1)$ of the leaves $x_1, \dots, x_m$, and that $T_2$ contains at least one leaf $x_j$.

Moreover, the root of $T_1$ has degree at most $\Delta-1$. By the inductive hypothesis, $T_1$ contains at least $\log_{\Delta-1}(m/(\Delta-1))+1 = \log_{\Delta-1}(m)$ distinct lengths of leaf-to-leaf paths between $0$ and $2(a-1)$, all witnessed by a some leaf $x_i$. Observe that the unique path from $x_i$ to $x_j$ has length $2a$. This gives $\log_{\Delta-1}(m) + 1$ lengths of paths between $0$ and $2a$, all witnessed by $x_i$. 

\medskip

\ninn \textbf{Case 2:} $\deg(r)=\Delta$. We again induct on the number of vertices of $T$. If $T$ has $\Delta+1$ vertices, then $m = \Delta$ and each leaf witnesses lengths 0 and 1, so the conclusion holds.

For the inductive step, again consider the two subcases outlined above. If one of the $T_i$'s contains all $m$ leaves $x_1, \dots, x_m$, then the claim follows by the inductive hypothesis applied to $T_i$. Otherwise, again like in Case 1 we may assume that $T_1$ has at least $m/\Delta$ leaves from the set $\{x_1, \dots, x_m\}$ and $T_2$ has at least one leaf $x_j$. Now the root of $T_1$ has degree at most $\Delta-1$, so we may use the slightly stronger bound obtained in Case 1 to find at least $\log_{\Delta-1}(m/\Delta)+1$ distinct lengths between $0$ and $2(a-1)$, all witnessed by some $x_i$. Together with the path of length $2a$ connecting $x_i$ to $x_j$, we obtain at least $\log_{\Delta-1}(m/\Delta)+2$ lengths of paths between $0$ and $2a$, all witnessed by $x_i$. 
\end{proof}

Our proof of \Cref{anylengths} proceeds by induction on the number of leaves in the tree $T$. After choosing a root appropriately, we either find many leaves at the same distance from it (and thus \Cref{lem:same_depth} applies), or instead find a subtree $T'$ with strictly smaller diameter but still having many leaves of $T$ (to which the inductive hypothesis applies). For the choice of root, we need the following well-known Helly-type lemma for trees (see, for instance,  \cite{GLhelly} or \cite{helly}).

\begin{lemma}\label{helly-tree}
    Let \(T\) be a tree and \(T_1,\hdots, T_s\) be a collection of subtrees of \(T\) such that $V(T_i)\cap V(T_j) \neq \emptyset$ for all $1\leq i<j\leq s$. Then $\cap_{i=1}^s V(T_i) \neq\emptyset$.
\end{lemma}

We are now ready to prove the main result of this section. 

\begin{proof}[{Proof of \Cref{anylengths}}]
The proof is by induction on $|L(T)|$. Note that the statement is trivial when $|L(T)|=1$, since there is one path length (namely zero), and when $|L(T)|\in [2,\Delta]$, since there are at least two path lengths in $T$ and $\log_{\Delta-1}(\Delta(\Delta-2)) \leq 2$.
Assume that the statement is true for all $\ell'<\ell$ and consider a tree $T$ with $\ell$ leaves. 
It is not hard to see that any two longest paths in \(T\)  share a vertex and thus~\Cref{helly-tree} implies there is a vertex $v$ which is contained in every longest path. Moreover, we may assume without loss of generality that $v$ is not a leaf, since otherwise its neighbour also satisfies this condition. Let $m$ be the length of the longest path in $T$. We consider two cases.

\medskip

\ninn \textbf{Case 1:} There is some leaf $x$ with $d(x,v)>m/2$.

Firstly, take a leaf $x$ that maximizes $d(x,v)$. Let $e=vu$ be the edge incident to $v$ on $T[v,x]$. Note that every leaf $y$ that is connected to $v$ by a path not containing $e$ satisfies $d(y,v) \leq m - d(x,v) < m/2$, as otherwise $T[x,y]=T[x,v]\cup T[v,y]$ would be a path of length greater than $m$. Moreover, since every longest path in $T$ passes through $v$, there must exist some leaf $y$ satisfying $e\notin T[y,v]$ and $d(y, v)=m-d(v, x)$.
It follows that every longest path in $T$ is formed by concatenating a path of length $d(x,v)$ from a leaf to $v$ (passing through $e$) together with a path of length $m - d(x,v)$ from $v$ to another leaf (avoiding $e$). 

Now, let $X_1$ be the set of leaves whose distance from $v$ is equal to $d(x,v)$ and let $X_2$ be the set of leaves whose distance from $v$ is equal to $m - d(x,v)$. $X_1$ and $X_2$ are clearly disjoint, and by the above, every longest path in $T$ goes from a vertex in $X_1$ to a vertex in $X_2$.

By relabelling if necessary, we may assume that $|X_1| \leq |X_2|$. Let $L=L(T)$ be the set of leaves in $T$, and observe that $|L \setminus X_1| \geq \ell/2$. We define $T'$ to be the smallest subtree of $T$ such that $L \setminus X_1 \subseteq V(T')$, and claim that $L(T')=L \setminus X_1$. Indeed, if $T'$ contained some other leaf $u \notin L \setminus X_1$, then $T' - u$ would still be connected and we would have $L \setminus X_1 \subseteq V(T' - u)$, a contradiction. Thus, $L(T')=L \setminus X_1\subseteq L(T)$, which implies that leaf-to-leaf paths in $T'$ are also leaf-to-leaf paths in $T$. Crucially, $V(T') \cap X_1 = \emptyset$ and thus the longest path in $T'$ is of length strictly less than $m$.

By the induction hypothesis, $T'$ contains leaf-to-leaf paths of at least 
$$\log_{\Delta-1}(\ell/2)+\log_{\Delta-1}(\Delta-2)\geq\log_{\Delta-1}\ell+\log_{\Delta-1}(\Delta-2)-1$$ distinct lengths, all 
strictly smaller than $m$. Together with the length $m$, we conclude that $T$ contains at least $\log_{\Delta-1}\ell+\log_{\Delta-1}(\Delta-2)$ distinct leaf-to-leaf path lengths.

\medskip

\ninn \textbf{Case 2:} The furthest leaf $x$ from $v$ satisfies $d(x,v) = m/2$.

In this case, every longest path is obtained by concatenating two internally vertex-disjoint paths of length $m/2$ from $v$ to different leaves. Let $X$ be the set of leaves of $T$ which are at distance precisely $m/2$ from $v$.  Now we split into two 
further subcases.

\medskip

\ninn \textbf{Case 2.1:} $|X|<(1-(\Delta-1)^{-2})\ell$. Consider the collection of subtrees of $T$ obtained by deleting the vertex $v$, and let $\overline T$ be one which contains at least $|X|/\Delta$ elements of $X$. 

Define $X' = X \setminus V(\overline{T})$, so that $|X'| \leq (1 - \Delta^{-1})|X|$. Recalling that $L$ is the set of leaves of $T$ and $|L|=\ell$, we define $T'$ to be the smallest subtree of $T$ such that $L \setminus X' \subseteq V(T')$. Using the same argument as in Case 1, it is easy to see that $L(T')=L \setminus X'$. Hence we have
\[
|L\setminus L(T')|=|X'|\leq 
\frac{\Delta-1}{\Delta}|X|
\leq \left( 1-\frac{1}{\Delta} - \frac{1}{\Delta(\Delta-1)} \right) \ell
=
(1-1/(\Delta-1))\ell.
\]
Thus, $T'$ has maximum degree at most $\Delta$, at least $\ell/(\Delta-1)$ leaves and by construction the longest path in $T'$ is strictly shorter than $m$ in length. Indeed, given a longest path, it has leaves $u_1$ and $u_2$, say, as endpoints. Supposing this path has length $m$, by the assumption of Case 2 above we know that $d(v, u_1)=d(v, u_2)=m/2$, and so $u_1, u_2\in X\setminus X'$. But then both $u_1$ and $u_2$ belong to the subtree $\overline T$, and hence the path connecting them doesn't pass through $v$, contradiction.

Thus, by the induction hypothesis, the leaf-to-leaf paths in $T'$ have at least $\log_{\Delta-1}((\Delta-2)\ell)-1$ many distinct lengths, and all of these also occur in $T$. Together with a leaf-to-leaf path of length $m$ in $T$, we get the required bound.

\medskip

\ninn \textbf{Case 2.2:} $|X|\geq (1-(\Delta-1)^{-2})\ell$. Then, it follows by applying \Cref{lem:same_depth} to $T$ rooted at $v$ that there are at least
\[\log_{\Delta-1}\left(\frac{(1-(\Delta-1)^{-2})\ell}{\Delta}\right)+2=\log_{\Delta-1}((\Delta-2)\ell)\]
distinct leaf-to-leaf path lengths, as required.
\end{proof}

\subsection{Paths of restricted length}\label{pfthm2} 

The aim of this section is to prove the lower bound of \Cref{lots}, which guarantees many lengths of short leaf-to-leaf paths in trees with not-too-small diameter and no vertices of degree 2. 

We will consider a path of maximum length in $T$ and look at its initial segment $P$ of length $N/2$. Each vertex $v$ in $P$ has a subtree hanging from it (which we root at $v$). We will split into two cases depending on the minimum root-to-leaf distance in each of these subtrees. If one of them is very deep, we will be able to find many short leaf-to-leaf paths inside of it; this is inspired by the approach of~\cite{NPSz}. Otherwise, all of the subtrees have shallow leaves and we will travel along $P$ to find many paths of distinct lengths connecting them. 

We will require the following classical result.

\begin{theorem}[Erd\H{o}s-Szekeres \cite{erdos1935combinatorial}]\label{erdos-szekeres-variant}
    Any sequence of \(n\) not necessarily distinct real numbers contains a monotone subsequence of length at least \(\sqrt{n}\).
\end{theorem}

We use \Cref{erdos-szekeres-variant} to prove the following lemma, which will be useful for proving \Cref{lots}~\emph{(i)}.

\begin{lemma}\label{additive} Let $(a_1, \dots, a_n)$ be a sequence of non-negative real numbers such that $a_i \leq m$ for each $1 \leq i \leq n$ and some $m>0$. Then
\[ \max\bigg\{\big|\{a_i+i: 1\leq i\leq n\} \big|, \big|\{a_i-i: 1\leq i\leq n\}\big| \bigg\}\geq \frac{n}{4\sqrt{m}}.\]
\end{lemma}
\begin{proof}

First, suppose that $m \leq n/2$. For each $1\leq i\leq n/(2m) $, set $A_i\coloneqq (a_j)_{j=2(i-1)m+1}^{(2i-1)m}$. \Cref{erdos-szekeres-variant} implies that each sequence $A_i$ contains a monotone subsequence of length at least $\sqrt{m}$. Let $B_i$ be the set of indices of this subsequence, so that $|B_i|\geq \sqrt{m}$ and $B_i \subseteq [2(i-1)m+1,(2i-1)m]$.

Let $X$ be the set of indices $1\leq k\leq \frac{n}{2m}$ for
which  \((a_i)_{i\in B_k}\)
is an increasing sequence, and set $Y\coloneqq \left[ \frac{n}{2m} \right] \setminus X$.
Suppose $|X| \geq \frac{n}{4m}$. 
    For each $k\in X$ and $i, j\in B_k$ with $i < j$, we have $a_{i} + i< a_{j} +j$, so the set
    \(
    A_k' = \{ a_i + i: i \in B_k \}
    \)
    consists of \(|B_k| \ge \sqrt{m}\) distinct elements.
    Moreover, given integers $1\leq k_1<k_2\leq \frac{n}{2m}$, for any $i_1\in B_{k_1}$ and $ i_2\in B_{k_2}$ we have
    \[
        a_{i_1}  + i_1
        \le m + (2k_1-1)m = 2k_1 m,
    \]
    and
    \[
    a_{i_2} + i_2 \ge 0 + 2(k_2-1)m +1
    \ge 2k_1 m +1,
    \]
    so the sets \(A_k'\) are pairwise disjoint.
    We conclude that 
    \[
    |\{ a_i+i : 1\leq i\leq n\}|
    \geq \sum_{k\in X} |A_k'|
    \geq \frac{n}{4m} \cdot \sqrt{m}=
    \frac{n}{4\sqrt{m}}.
    \]
    If instead we have \(|X|<\frac{n}{4m}\), then \(|Y| \ge \frac{n}{4m}\), and for every \(k\in Y\), \((a_i)_{i\in B_k}\) is a decreasing subsequence.
    An analogous argument shows that in this case \(|\{a_i - i: 1\leq i\leq n\}|\ge \frac{n}{4\sqrt{m}}\).

    If $m>n/2$, \Cref{erdos-szekeres-variant} guarantees that the sequence $(a_i)_{i=1}^n$ has a monotone subsequence of length at least $\sqrt{n}$. If this sequence is increasing, then $|\{a_i+i:1\leq i\leq n\}|\geq \sqrt{n}$, while if the sequence is decreasing, then 
    $|\{a_i-i: 1\leq i\leq n\}|\geq \sqrt{n} $,
    and note that both quantities are at least \( \frac{n}{4\sqrt m}\), as required.
\end{proof}

\begin{proof}[{Proof of \Cref{lots}(i)}] 
We can assume that $N$ is an even integer.
Let $P = v_0 v_1 \dots v_M$ be a path of maximum length in $T$ and let $P'= v_0 v_1 \dots v_{N/2}$ be its initial segment of length $N/2$. For each $1 \leq i \leq N/2$, let $T_i$ be the connected component of $T \setminus E(P)$ that contains $v_i$. 

Observe that for every $1\leq i\leq N/2$ and every leaf $x\in T_i\setminus\{v_i\}$, we must have $d(x, v_i)\leq N/2$, as otherwise we would have $d(x, v_M)>M$, a contradiction.

\medskip

\ninn \textbf{Case 1:} There exists some $1\leq i\leq N/2$ such that for every leaf $x\in T_i\setminus\{v_i\}$, we have $d(x, v_i)>N^{2/3}/2$.
Then \(v_i\) has a neighbour \(u_i \in V(T_i)\) which is not a leaf and hence has degree at least 3 in \(T_i\).
Let \(T'\) be a maximal binary subtree of $T_i - v_i$ rooted at \(u_i\), and note that every leaf of \(T'\) is also a leaf of \(T\). Every leaf of \(T'\) is at distance at least $N^{2/3}/2 - 1$ from $u_i$. Together with the fact that each non-leaf vertex in $T'$ has two children, this implies that
$T'$ contains at least \(2^{N^{2/3}/2 -1}\) leaves.
As established above, each of these leaves is at distance at most $N/2$ from $v_i$. Thus, there exists some $1\leq d\leq N/2$ 
for which at least $2^{N^{2/3}/2}/N\geq 2^{N^{2/3}/3}$ distinct leaves in $T_i$ 
are all at distance precisely $d$ from $v_i$. By \Cref{lem:same_depth} we can then find a leaf $x\in T_i$ witnessing at least $\log (2^{N^{2/3}/3}/3)+ 2 \geq N^{2/3}/3$ distinct leaf-to-leaf path lengths in $T$, and all of these lengths are at most equal to $2d\leq N$.

\medskip

\ninn \textbf{Case 2:} For every $1\leq i\leq N/2$ there exists a leaf $x_i\in T_i$, $x_i\neq v_i$, such that $a_i \coloneqq d(x_i, v_i) \leq N^{2/3}/2$.

Observe that the set of path lengths connecting pairs in $\{x_1, \dots, x_{N/2}\}$ is precisely
$$X = \{a_i + a_j + j - i : 1 \leq i < j \leq N/2\}.$$
Moreover, any $(x_i, x_j)$-path has length at most $N/2 + N^{2/3} \leq N$. By applying \Cref{additive} with $m = N^{2/3}/2$, we see that 
$$\max\left( \big|\{a_i + i:1\leq i\leq N/2\}\big|, \big|\{a_i - i: 1\leq i\leq N/2\}\big|\right) \geq 
\frac{N^{2/3}}{4\sqrt{2}}.$$
If the inequality holds for $\{a_i + i: 1\leq i\leq N/2\}$, then 
$$|X| \geq |\{a_1 - 1 + (a_i + i) : 2 \leq i \leq N/2\}| \geq N^{2/3}/6,$$
with $N^{2/3}/6$ distinct lengths being witnessed by $x_1$. If it holds for $\{a_i - i: 1\leq i\leq N/2\}$, then 
$$|X| \geq |\{a_{N/2} + (N/2) + (a_i - i): 1 \leq i \leq (N/2)-1\}| \geq N^{2/3}/6,$$
with $x_{N/2}$ witnessing all these lengths, as desired.\end{proof}

\section{Trees with few leaf-to-leaf path lengths}\label{constructions}

In this section we prove \Cref{thm:upper-bound-lengths} and the second part of \Cref{lots}. Each result is obtained by taking a sequence $(a_i)$ with a suitable additive structure and constructing an $n$-vertex tree $T_n((a_i))$ from it. We first describe the general construction, and then provide a suitable choice of $(a_i)$ for each of the two results.

\subsection{The general construction}\label{sect:constr}
Let  \(n \geq 4\) be even. Let \(m\in\mathbb{N}\) and consider a positive integer sequence \((a_i)_{i=1}^m\). We will now describe a general construction of an \(n\)-vertex 1--3 tree \(T_n((a_i))\) based on this sequence. For the most part, our construction consists of a path together with a collection of perfect binary trees attached to the path's internal vertices, with the sequence $(a_i)$ dictating the depths of the perfect trees.

Consider the periodic sequence $(a'_i)_{i \geq 1}$ given by 
$$
a_1, \dots, a_{m}, a_1, \dots, a_{m}, a_1, \dots,
$$
and take its shortest initial segment $(a'_1, \dots, a'_t)$ with the property that $S \coloneqq 2+\sum_{i=1}^{t} 2^{a'_i} \geq n$. Note that $t \geq 1$. Based on our choice of $t$ and the fact that $n$ and $S$ are even, it must be the case that $S - 2^{a'_t} \leq n - 2$.

We will now describe how to construct $T_n((a_i))$. We start with a path $P = v_0 v_1 \dots v_{t+1}$. For each $i \in [t-1]$, we take a perfect binary tree $(T_i, r_i)$ on $a'_i$ layers, and add an edge from $v_i$ to $r_i$. Thus far, every vertex in the tree other than $v_t$ has degree either 1 or 3 and the total number of vertices is 
$$t+2 + \sum_{i=1}^{t-1} (2^{a'_i} - 1) = S - 2^{a_t'} + 1 \leq n -1.$$

Let $L = n - (S - 2^{a'_t} + 1) \geq 1$, which must be odd since $n$ and $S$ are even. Since $S \geq n$, we have that $L \leq 2^{a'_t} - 1$. We take a perfect binary tree $(\tilde{T}, r_t)$ on $\lceil \log (L + 1) \rceil \leq a'_t$ layers. With this choice, we have $L \leq |V(\tilde{T})| < 2L$. We now proceed to iteratively delete pairs of leaves sharing a parent from the lowest layer of $\tilde{T}$, until we obtain a tree $T_t$ which has precisely $L$ vertices (which is possible since both \(L\) and \(|V(\tilde{T})|\) are odd). By removing pairs of leaves which share a parent, and always from the lowest layer, we guarantee that the resulting $T_t$ is still a binary tree, with its leaves spanning at most two layers. Adding an edge from $r_t$ to $v_t$ then completes the construction of $T_n((a_i))$. 
Observe that for any two leaves \(x_i \in T_i\), \(x_j \in T_j\) with \(i \neq j\),
the unique path from \(x_i\) to \(x_j\) consists of the path inside \(T_i\) from \(x_i\) to \(r_i\), the edge $r_iv_i$, the path from \(v_i\) to \(v_j\) in \(P\), the edge $v_jr_j$ and finally the path from \(r_j\) to \(x_j\); cf.~\Cref{fig1}.

\begin{figure}[htb]
\captionsetup{justification=centering}
\centering
\begin{tikzpicture}[scale=1, cross/.style={path picture={ 
    \draw[black]
;
}}]

\fill[] (-7, 0) circle(0.05) node[left] {\small$v_0$};
\fill[] (6.3, 0) circle(0.05) node[right] {\small$v_{t+1}$};
\fill[] (-6.7, 0) circle(0.05) node[above] {\small$v_1$};
\fill[] (6, 0) circle(0.05) node[above] {\small$v_t$};
\fill[] (6, -0.3) circle(0.05) node[right] {\small$r_t$};
\draw[thick] (6, 0) -- (6, -0.3);
\fill[] (0, 0.3) node[] {$P$};

        \draw[] (-7, 0) -- (6.3, 0);

        \fill[] (-6.7, 0) circle(0.05) node[above] {\small$v_1$};
        \fill[] (-7.2, -1) node[below] {\small$a_1$};
        \fill[] (-6.7, -0.3) circle(0.05) node[right] {\small$r_1$};
        \draw[thick] (-6.7, 0) -- (-6.7, -0.3);
        \draw[thick,black,pattern=north east lines,pattern color=gray] (-6.7, -0.3) -- (-6.5, -1.7) -- (-6.9, -1.7) -- (-6.7,-0.3);
        \fill[] (-6.7, -2) node[] {\small$T_1$};

        \fill[] (-5.2, 0) circle(0.05) node[above] {\small$v_2$};
        \fill[] (-5.7, -1) node[below] {\small$a_2$};
        \fill[] (-5.2, -0.3) circle(0.05) node[right] {\small$r_2$};
        \draw[thick] (-5.2, 0) -- (-5.2, -0.3);
        \draw[thick,black,pattern=north east lines,pattern color=gray] (-5.2, -0.3) -- (-5, -2) -- (-5.4, -2) -- (-5.2,-0.3);
        \fill[] (-5.2, -2.3) node[] {\small$T_2$};

        \fill[] (-4.2, -1) node[] {$\dots$};

        \fill[] (-3, 0) circle(0.05) node[above] {\small$v_m$};
        \fill[] (-3.5, -1) node[below] {\small$a_m$};
        \fill[] (-3, -0.3) circle(0.05) node[right] {\small$r_m$};
        \draw[thick] (-3, 0) -- (-3, -0.3);
        \draw[thick,black,pattern=north east lines,pattern color=gray] (-3, -0.3) -- (-2.8, -2.3) -- (-3.2, -2.3) -- (-3,-0.3);
        \fill[] (-3, -2.6) node[] {$T_m$};

        \fill[] (-1.5, 0) circle(0.05) node[above] {\small$v_{m+1}$};
        \fill[] (-2, -1) node[below] {\small $a_{1}$};
        \fill[] (-1.5, -0.3) circle(0.05) node[right] {\small$r_{m+1}$};
        \draw[thick] (-1.5, 0) -- (-1.5, -0.3);
        \draw[thick,black,pattern=north east lines,pattern color=gray] (-1.5, -0.3) -- (-1.3, -1.7) -- (-1.7, -1.7) -- (-1.5,-0.3);
        \fill[] (-1.5, -2) node[] {\small$T_{m+1}$};
        
        \fill[] (0, -1) node[] {$\dots$};

        \fill[] (1.5, 0) circle(0.05) node[above] {\small$v_{2m}$};
        \fill[] (1, -1) node[below] {\small $a_{m}$};
        \fill[] (1.5, -0.3) circle(0.05) node[right] {\small$r_{2m}$};
        \draw[thick] (1.5, 0) -- (1.5, -0.3);
        \draw[thick,black,pattern=north east lines,pattern color=gray] (1.5, -0.3) -- (1.3, -2.3) -- (1.7, -2.3) -- (1.5,-0.3);
        \fill[] (1.5, -2.6) node[] {\small$T_{2m}$};

        \fill[] (2.8, -1) node[] {$\dots$};

        \draw[thick,black,pattern=north east lines,pattern color=gray] (6, -0.3) -- (6.2, -1.4) -- (5.8, -1.4) -- (6,-0.3);
        \fill[] (6, -1.7) node[] {$T_t$};
        
\end{tikzpicture}
\caption{The construction of the tree $T_n((a_i))$. Each subtree $T_{i+mj}$ for $j\geq 1$, except $T_t$, represents a perfect binary tree on $a_i$ layers, whose root neighbours the corresponding vertex on the horizontal path $P$. Note that this pattern repeats cyclically every $m$ steps. To the vertex $v_t$, we instead append the specific tree $T_t$, as described in the context.} \label{fig1}
\end{figure}
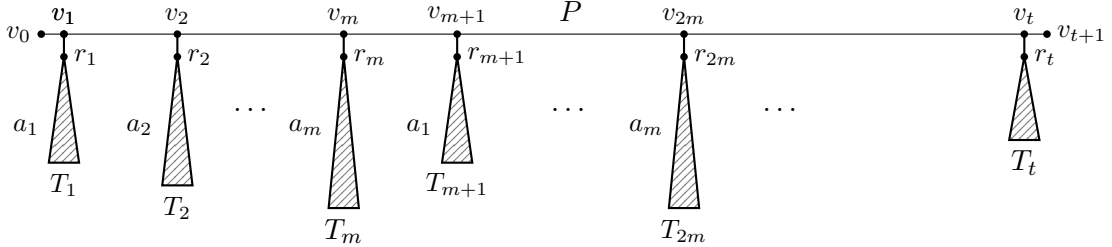

\subsection{Upper bound on path lengths in \([N]\)}
\newcommand{\abs}[1]{\left|#1\right|}
\newcommand{\ceil}[1]{ \lceil #1 \rceil }

For a set \(U\subseteq \mathbb{Z}\) and \(k\in \mathbb{Z}\), let \(k\cdot U \coloneqq \{ ku: u \in U\}\) and \(k + U = U + k \coloneqq \{u+k : u \in U\}\).

\begin{proposition}\label{prop:additive-to-tree}
    Suppose that there exist a positive integer \(m\), sets \(U,V \subseteq \mathbb{Z}\) and a real \(\beta \in (0,1)\) that satisfy the following:
    \begin{enumerate}
        \item \(U+V \subseteq [m] \subseteq U-V\); and
        \item \(\abs{U+V} = m^\beta\).
    \end{enumerate}
    Let \(M = \lfloor m^{2-\beta}\rfloor\) and \(n\ge M\) even.
    If $m^{1-\beta}\geq 4$, then there exists a 1--3 tree \(T\) on \(n\) vertices such that the number of distinct leaf-to-leaf path lengths in \([M]\) is at most 
    \(26M^{\frac{1}{2-\beta}}\).
\end{proposition}
\begin{proof}
    Since \(U-V = [m]\), we can find a sequence of pairs \((u_i,v_i)_{i=1}^m\) such that 
    \(i = u_i - v_i\).
    Consider the sequence \((a_i)_{i=1}^m\), defined by \(a_i = u_i + v_i\). Let \(T\) be the tree $T_n((a_i))$ as defined in \Cref{sect:constr}.
    
    We proceed to count the number of leaf-to-leaf path lengths at most $M$ in $T$. We will prove that there are at most $13m$ such paths, which suffices to prove the proposition since \[M^{\frac{1}{2-\beta}} \geq (m^{2-\beta} - 1)^{\frac{1}{2-\beta}} \geq \frac{m}{2^{\frac{1}{2-\beta}}} \geq \frac{m}{2},\]
    where we used $m^{1-\beta} \geq 4$ in the second inequality.
    
    First note that, for two leaves $u$ and $v$ belonging to the same subtree $T_i$, say, we must have $d(u, v)\leq 2m$. It therefore suffices to show that there are at most $11m$ lengths arising when we consider leaves belonging to different subtrees, say $u\in T_i$ and $v\in T_j$ with $1\leq i<j\leq t$ and $d(u, v)\leq M$, or when $u=v_0$ or $v=v_{t+1}$.
    
    Recall that we write $(i)_m^*$ for the integer in $\{1, 2, \dots m\}$ congruent to $i \mod m$. Write \(i = (i)_m^* + \ell_i \cdot m\), and \(j = (j)_m^* + \ell_j \cdot m\).

    \medskip

    \ninn{\bf Case 1.} If $u\neq v_0, v\neq v_{t+1}$ and $j\neq t$, then \(d(u, v)\) is precisely equal to
    \[
    a_{(i)_m^*} + j - i + a_{(j)_m^*} = 
    a_{(i)_m^*} + (j)_m^* - (i)_m^* + a_{(j)_m^*} + (\ell_j -\ell_i)m
    = 2 u_{(j)_m^*} + 2 v_{(i)_m^*} + (\ell_j - \ell_i)m,
    \]
    where we have used the fact that $a_i+i=2u_i$ and $a_i-i=2v_i$ for all $i$.
    Since \(d(u, v) \in [M]\), we must have $0\leq \ell_j-\ell_i\leq \lceil M/m\rceil$ and thus
    \[
    d(u, v) \in 2\cdot (U+V) + m \cdot \{0, \dots, \ceil{M/m}\} \eqqcolon A.
    \]
    
    But note that $|A|\leq |U+V| \cdot (2M/m)\leq 2m$, so there are at most $2m$ distances we can find in this case.

    \medskip

    \ninn{\bf Case 2.} If $u\neq v_0$ and $j=t$, then let \(c\in \mathbb{N}\) be such that the leaves in \(T_t\) are all at distance either \(a_{(t)_m^*}-c\) or \( a_{(t)_m^*}-c-1\) from the root.
    
    We then have that \(d(u, v)\) is precisely either
    \[
        a_{(t)_m^*} -c + t-i + a_{(i)_m^*}
    \]
    or 
    \[
        a_{(t)_m^*} -c-1 + t-i + a_{(i)_m^*},
    \]
    i.e. we have $d(u, v)\in (A-c)\cup (A-c-1)$, and so we obtain at most $4m$ distances in this case.

    \medskip

    \ninn{\bf Case 3.} If $u=v_0$, then the cases of $v=v_{t+1}$ or $j=t$ provide at most three new distances. If we instead have $j\neq t$, then $j\leq M$ since $d(u, v)\leq M$, and so $0\leq \ell_j\leq \lceil M/m \rceil$. Thus,
    \[
    d(u, v)=a_{(j)_m^*}+j+1=2u_{(j)_m^*}+1+\ell_j m\in 2\cdot U + 1+ m\cdot\{0, 1, \dots, \lceil M/m\rceil \}
    \]
    which is a set of size at most $|U+V|\cdot 2M/m\leq 2m$.

    \medskip
    
    \ninn {\bf Case 4.} If $v=v_{t+1}$, then the case $i=t$ provides at most two new distances. Assuming that $u\neq v_0$ and $i\neq t$, we have that $i\geq (t+1)-M$ since $d(u, v)\leq M$, and by proceeding similarly to the previous case we have again at most $2m$ new distances.

    Putting everything together, we see that indeed $d(u, v)$ can take at most $10m+5\leq11m$ distinct values when $u$ and $v$ do not lie in the same subtree, which completes the proof.
\end{proof}

\begin{proof}[{Proof of \Cref{thm:upper-bound-lengths}}]
    Given $N$ in the statement of the theorem, it is clear that we may assume $n>20N$, say, as for smaller $n$ the conclusion follows by considering the almost-perfect tree on $n$ vertices, which only has about $\log n$ leaf-to-leaf path lengths in total. Let $k\in\mathbb{N}$ be the smallest integer such that $N \leq (169/10)^k$. Set
    \[
    X = \{1,2,5,7\} \text{ and } Y = \{-5,-4,-1,1 \}
    \]
    and observe that
    \[
    X-Y = [0,12] \text{ and } X+Y = \{ -4,-3,-2,0,1,2,3,4,6,8 \}.
    \]
    We further set
    \[
        U = \left\{ \sum_{i=0}^{k-1} x_i 13^i : x_i \in X \right \} + \frac{13^k-1}{6} +1
    \]
    and
    \[
        V = \left\{ \sum_{i=0}^{k-1} y_i 13^i : y_i \in Y \right\} + \frac{13^k-1}{6}.
    \]
    Observe that \(U-V = [13^k]\) and \(\abs{U+V} = 10^k\).
    Thus for any even $n\geq (169/10)^k$, applying \Cref{prop:additive-to-tree} with $m=13^k$ and \(\beta = \log 10 / \log 13\) (and hence $M=\lfloor(169/10)^k\rfloor$) gives a tree $T$ on $n$ vertices with at most $26M^{\frac{1}{2-\beta}}$ leaf-to-leaf path lengths in $[M]$. In particular, as $10M/169<N\leq M$, there are at most $500N^{\frac{1}{2-\beta}}$ lengths in $[N]$, and so the conclusion of the theorem follows.
\end{proof}

Similar constructions to those in the above proof can be found in the work of Ruzsa~\cite{ruzsa}. 

\subsection{Upper bound on path lengths witnessed by a leaf}

\begin{proof}[{Proof of \Cref{lots}(ii)}]
We will provide an explicit construction of an $n$-vertex 1--3 tree in which each individual leaf witnesses at most $20N^{2/3}$ distinct leaf-to-leaf path lengths between 0 and $N$.

Let $m \coloneqq \lfloor N^{1/3} \rfloor$. Recall that we write $(i)_m$ for the residue of $i \mod m$, considered as an element of $\{0, 1,\dots, m-1\}$, and define the sequence $(a_1, \dots, a_{m^2})$ by
\[a_i \coloneqq  \left\lceil\frac{i}{m} \right\rceil \cdot m - (i-1)_m .\]
Observe that $1 \leq a_i \leq m^2 \leq N^{2/3}$ for each $i \in [m^2]$. 
Consider the tree \(T = T_n((a_i)_{i=1}^{m^2})\) described in \Cref{sect:constr}.

We claim that $T$ satisfies the conditions of the theorem. Suppose for the sake of contradiction that there is a leaf $u \in V(T)$ witnessing more than $20N^{2/3}$ distinct lengths in $[0,N]$. 
Then \(u\) witnesses at least \(18N^{2/3}\) distinct lengths in $[2N^{2/3}, N]$.
We will show how to handle the case when $u \in T_{j_0}$ for some $j_0\in[t]$, since the case when $u \in \{v_0, v_{t+1}\}$ is only easier, as it will be clear by the end of the proof.
Set $q = \lfloor18N^{2/3}\rfloor$ and let $s_1, \dots, s_q$ be leaves such that the distances $d(u, s_i)$ are all distinct and in the interval $[2N^{2/3}, N]$. 

Since \(T_{j_0}\) has at most \(N^{2/3}\) layers, every leaf-to-leaf path in \(T_{j_0}\) is of length at most \(2N^{2/3}-2\). But for every \(s_i\) we have \(d(s_i,u) \ge 2N^{2/3} \), and thus \(s_i \notin T_{j_0}\) for all \(i\).

For \(j\neq j_0, t\), any two leaves in \(T_j\) clearly are at the same distance from \(u\), since \(T_j\) is a perfect binary tree;
and, provided \(j_0 \neq t\), leaves in \(T_t\) can have at most two distinct distances to \(u\), since leaves in \(T_t\) are spread over at most two layers. Moreover, the only leaves not in any tree \(T_i\) are \(v_0, v_{t+1}\). Therefore, after relabeling the leaves \(s_i\) if necessary, we may assume that for \(1\le i \le q-4\),
there exists \(j_i \in [t]\setminus \{ j_0, t\}\) with
\(s_i \in T_{j_i} \), and the indices \(j_{i}\) are pairwise distinct.

For each integer $0 \leq k \leq t/m^2$, define $I_k\coloneqq \{km^2+1,\dots, (k+1)m^2\}$. Let $k_0$ satisfy $I_{k_0} \ni j_0$. For each $i \in [q-4]$, if $j_i \in I_k$ then we must have $|k-k_0| < 2N^{1/3}$ since $d(s_i, u) \leq N$. Then, by pigeonhole there exists $k$ such that

\[|I_k \cap \{j_i : i \in [q-4]\}| \geq \frac{q-4}{4N^{1/3}} \geq 4N^{1/3} \geq 4m.\]

We split $I_k$ into $I_L = I_k \cap [0, j_0)$ and $I_R = I_k \cap (j_0, t-1]$, and observe that both \(I_L\) and \(I_R\) are non-empty if and only if \(k=k_0\).

Recall that  \(T_{j_i}\) is a perfect binary tree on \(a_{(j_i)_{m^2}^*}\) layers. For every \(i \in [q-4]\) with \(j_i \in I_R\), we have $j_i > j_0$ and thus
\begin{equation}\label{eq:distances}
d(u, s_i) = a_{(j_0)_{m^2}^*} + j_i - j_0+ a_{(j_i)_{m^2}^*} 
= a_{(j_0)_{m^2}^*} + (j_i)_{m^2}^* + km^2 - j_0+ a_{(j_i)_{m^2}^*},
\end{equation}
since the distance between \(u\) and \(v_{j_0}\) in \(T_{j_0}\) is \(a_{(j_0)_{m^2}^*}\),
the distance between \(v_{j_0}\) and \(v_{j_i}\) in \(P\) is \(j_i - j_0\), and the distance between \(v_{j_i}\) and \(s_i\) in \(T_{j_i}\) is \(a_{(j_i)_{m^2}^*}\). However, from the definition of \(a_{(j_i)_{m^2}^*}\) it easily follows that \(a_{(j_i)_{m^2}^*} + (j_i)_{m^2}^* \equiv 1 \pmod m\), which implies that the RHS of \eqref{eq:distances} can take at most $m$ distinct values as $j_i \in I_R$ varies. Hence we must have \(|I_R| \leq m\), which implies \(|I_L| \ge |I_k| - m \geq 3m\).

Similarly, for $j_i \in I_L$ we have $j_i < j_0$ and thus
\begin{equation}\label{eq:dist2}
d(u, s_i) = a_{(j_i)_{m^2}^*} + j_0 - j_i + a_{(j_0)_{m^2}^*}
 = a_{(j_i)_{m^2}^*} + j_0 - (j_i)_{m^2}^* - km^2 +  a_{(j_0)_{m^2}^*}.
\end{equation}
However, for each $s \in [m^2]$ we see from the definition of \(a_{s}\) that 
\[-m+1 \leq \left(\frac{s}{m}\cdot m - (s-1)_m\right) -s  \leq a_s - s \leq \left(\left(\frac{s}{m} + 1\right)\cdot m - (s-1)_m\right)- s \leq m. \]
This implies that the RHS of \eqref{eq:dist2} can take at most $2m$ distinct values as $j_i \in I_L$ varies. Together with the fact that $|I_L| \geq 3m$, this yields the desired contradiction. 

It is not hard to see that when \(u \in \{v_0,v_{t+1}\}\) essentially the same argument again gives a contradiction.
\end{proof}

\section{Cycles in degree $k$-critical graphs}\label{sect:cycles}
Recall that an $n$-vertex graph is \emph{degree $k$-critical} for some $k \geq 3$ if it has $(k-1)n - {k \choose{2}} + 1$ edges and no proper induced subgraph with minimum degree at least $k$. In this section, we prove a lower bound on the number of cycle lengths in graphs belonging to a general family that contains all degree $k$-critical graphs (i.e. Theorem~\ref{thm:ord->uppc(G)} below). 
By taking $k = 3$, this result implies Theorem~\ref{thm:d-3-cCL}.

Our first lemma provides a useful ordering of the vertex set of a degree $k$-critical graph; we remark that the case $k=3$ was already proven in \cite[Lemma 1]{erdos1988cycles}.
Let $\mathcal{X}=x_1,x_2,\dots,x_n$ be a given ordering of the vertex set $V$ of a graph $G$.
For $x_i\in V$, define $N_\mathcal{X}^+(x_i)=\{x_j\in N_G(x_i):i<j\}$ and $N_\mathcal{X}^-(x_i)=\{x_j\in N_G(x_i):i>j\}$.
We also define $d_\mathcal{X}^+(x_i)=|N_\mathcal{X}^+(x_i)|$ and $d_\mathcal{X}^-(x_i)=|N_\mathcal{X}^-(x_i)|$.
We will generally omit the subscript $\mathcal{X}$ if the ordering is clear from  context.

\begin{lemma}\label{lem:orien exist}
    Let $k\geq 3$ and $n \geq k+1$. Given any $n$-vertex degree $k$-critical graph $G$, there exists an ordering $\mathcal{X}=x_1, x_2,\dots,x_n$ of $V=V(G)$ such that
    \begin{align*}
    d^{+}(x_i)=\left\{
    \begin{array}{rcl}
         k && \text{if } i=1, \\
         k-1 && \text{if } i\in [2, n-k+1],\\
         n-i && \text{if } i\in [n-k+2, n].
    \end{array}
    \right.
\end{align*}
\end{lemma} 

\begin{proof}
    We construct the ordering $x_1, x_2,\dots,x_n$ iteratively. As a first step, note that by definition there exists a vertex $x_1\in V$ satisfying $d_G(x_1)\leq k$; otherwise, deleting any vertex in $G$ would leave a proper induced subgraph with minimum degree at least $k$.
    
    Assume we have chosen $\{x_1,\dots,x_\ell\}$  for some $\ell \in [n-k]$. 
    Since the minimum degree of the proper induced subgraph $G[V\setminus\{x_1,\dots,x_\ell\}]$ is less than $k$, there exists a vertex $v\in V \setminus \{x_1,\dots,x_\ell\}$ such that $|N_G(v)\setminus \{x_1,\dots,x_\ell\}|\leq k-1$. 
    Define $x_{\ell+1}=v$.

    After selecting $\{x_1,\dots,x_{n-k+1}\}$, we order the remaining $k-1$ vertices arbitrarily as $x_{n-k+2},\dots,x_n$. 
    Then the ordering $x_1, x_2,\dots,x_n$ satisfies $d^+(x_1)\leq k$, $d^+(x_i)\leq k-1$ for $i\in [2,n-k+1]$, and $d^+(x_i)\leq n-i$ for $i \in [n-k+2,n]$.

    We can thus bound the number of edges in $G$ as 
\[
    |E(G)|=\sum_{i=1}^nd^{+}(x_i) \leq k+(k-1)(n-k)+\sum_{i=0}^{k-2}i=n(k-1) - \frac{k(k-1)}{2} + 1.
\]
    By definition, $G$ has exactly $n(k-1) - \frac{k(k-1)}{2} + 1$ edges, hence all inequalities in the previous expression must hold with equality, which proves the lemma.
\end{proof}

Fix an integer $k\geq 3$.
Let $G$ be a graph on $n$ vertices and let $\mathcal{X}=x_1, x_2,\dots,x_n$ be an ordering of $V(G)$.
We say that $(G,\mathcal{X})$ is a \textit{k-ordered graph} 
if
\begin{itemize}
    \item[(1)] $x_{n-1}x_n\in E(G)$,
    \item[(2)] $d^+(x_i)\in [2,k]$ for $i\in [1,n-2],$ and
    \item[(3)] $d^-(x_i)\geq 1$ for $i\in [2,n].$
\end{itemize}
Suppose $G$ is a degree $k$-critical graph on $n$ vertices
and let $\mathcal{X}=x_1, x_2,\dots,x_n$ be the ordering given by Lemma~\ref{lem:orien exist}. 
Then it is easy to verify that $(G,\mathcal{X})$ is a $k$-ordered graph.

Given a graph $G$, we use $\mathcal{C}_G$ to denote the set of cycle lengths in $G$. We can now state the main result of this section.

\begin{theorem}\label{thm:ord->uppc(G)}
    Let $k\geq 3$ and $n\geq k+1$. 
    If $(G,\mathcal{X})$ is a $k$-ordered graph on $n$ vertices, then $|\mathcal{C}_G|\geq \frac{\log n}{3+\log k}-2$.
\end{theorem}

Throughout the rest of this section, we will assume that $(G,\mathcal{X})$ is a $k$-ordered graph. 
Let $u,v\in V(G)$ and $P=w_1w_2\cdots w_t$ be a path in $G$ where $w_1=u$ and $w_t=v$.
We call $P$ a \textit{forward (u,v)-path} if $w_{i+1}\in N_\mathcal{X}^+(w_i)$ for every $i\in [t-1]$. In particular, we also view a path consisting of a single vertex as a forward path.

Towards the proof of Theorem~\ref{thm:ord->uppc(G)},
we start with a series of lemmas.
The first lemma establishes a lower bound on $|\mathcal{C}_G|$ based on the length of the longest forward path in $G$.
\begin{lemma}\label{lem:vine}
    Let $k \geq 3$ and let $(G,\mathcal{X})$ be a $k$-ordered graph.
    For any integer $\ell\geq 2$, if $G$ contains a forward path of length $\ell$, then $|\mathcal{C}_G|\geq \log (\ell+1) -1$.
\end{lemma}

\begin{proof}
Fix a vertex $v_1\in V(G)\setminus \{x_{n-1},x_n\}$.
Let $P=v_1\cdots v_t$ be a longest forward path starting at $v_1$. Note that $d^+(v_1) \geq 2$, and thus $v_1$ has a forward neighbour $v' \in V \setminus \{x_n\}$. Since each $v \in V \setminus \{x_n\}$ satisfies $d^+(v) \geq 1$, there exists a forward path from $v'$ to $x_n$. For the same reason, each longest forward path has $x_n$ as its endpoint. Thus, $v_t = x_n$ and $t \geq 3$.

We claim that $\mathcal{C}_G\cap[t,2t-2]\neq \emptyset$. We will construct a cycle of suitable length by following a strategy similar to~\cite{bollobas1989long}. Given two vertices $a, b \in V(P)$, we write $a < b$ if $a$ precedes $b$ in $P$, and we write $a \leq b$ if either $a < b$ or $a = b$. Given a path $Q$ and vertices $u, v \in Q$, recall that $Q[u,v]$ denotes the unique subpath of $Q$ whose endpoints are $u$ and $v$. Following the idea in~\cite{bondy1981relative}, we define a slightly stronger version of \emph{vine} based on $P$ as a collection of internally vertex-disjoint forward paths $\mathcal{Q} = \{Q_i: i \in [m]\}$ such that the ends of $Q_i$ are $(a_i, b_i)$ and the following are satisfied:
\begin{enumerate}[label = (\arabic*)]
    \item $V(Q_i)\cap V(P)=\{a_i,b_i\}$ and $\ell(P[a_i,b_i])\geq 2$ for every $i\in [m]$;\label{ends}
    \item $v_1=a_1<a_2<b_1\leq a_3<b_2\leq a_4<b_3\leq \cdots\leq a_m<b_{m-1}<b_m=x_n$; and\label{oder-2}
    \item $a_{i+1}$ is the immediate predecessor of $b_i$ on $P$ for every $i\in [m-1].$\label{predec}
\end{enumerate}

We will first show the existence of the above structure $\mathcal{Q}$ based on $P$ and then argue that this implies the existence of a cycle of the desired length. For the first of these tasks, we argue inductively that we can construct a collection of paths $\mathcal{Q}$ satisfying \ref{ends}, \ref{oder-2}, and \ref{predec}, and then show that satisfying these conditions implies that the paths are internally vertex-disjoint. 

Suppose that for some $r \geq 0$ we have constructed paths $Q_1, \dots, Q_r$ satisfying \ref{ends}, \ref{predec}, as well as
\begin{enumerate}[label = (\arabic*$'$)]
\setcounter{enumi}{1}
    \item $v_1 = a_1 < a_2 < b_1 \leq a_3 < b_2 \leq a_4 < b_3 \leq \dots \leq a_r < b_{r-1} < b_r$. \label{oder-2'}
\end{enumerate}

Let us show how to construct $Q_{r+1}$. If $r = 0$, we let $a_{r+1} = v_1$, and observe that with this choice we have $d^+(a_{r+1}) \geq 2$. If $r > 0$, then we may assume that $b_r < x_n$ as otherwise \ref{oder-2} is also satisfied and we are done. In this case, we let $a_{r+1}$ be the immediate predecessor of $b_r$ on $P$, and observe that again $d^+(a_{r+1}) \geq 2$ since $a_{r+1} < b_r < x_n$. 

Since $d^+(a_{r+1}) \geq 2$, $a_{r+1}$ has a neighbour $c_{r+1} \in N^+(a_{r+1}) \setminus \{b_r\}$. Let $P_{r+1}$ be a forward $(c_{r+1},x_n)$-path, and let $b_{r+1}$ be the vertex in $V(P_{r+1}) \cap V(P)$ which minimizes $\ell(P_{r+1}[c_{r+1}, b_{r+1}])$. Indeed, $x_n \in V(P_{r+1}) \cap V(P)$ and thus such a vertex must exist. Define \[Q_{r+1} = \{a_{r+1}c_{r+1}\} \cup P_{r+1}[c_{r+1}, b_{r+1}],\] so that $Q_{r+1}$ is a forward path. By definition, $V(Q_{r+1}) \cap V(P) = \{a_{r+1}, b_{r+1}\}$. Moreover, $\ell(P[a_{r+1}, b_{r+1}]) \geq 2$, since otherwise $a_{r+1}b_{r+1} \in E(P)$, and as $Q_{r+1}$ is a forward path it would follow that $b_{r+1}=b_r\neq c_{r+1}$, and thus $(P \setminus \{a_{r+1}b_{r+1}\}) \cup Q_{r+1}$ is a longer forward path starting at $v_1$, contradiction. Finally, if $r \geq 2$, we have $b_{r-1} \leq a_{r+1}$, since $a_r$ is the predecessor of $b_{r-1}$ and $a_{r+1}$ is the predecessor of $b_r$. This shows that $Q_1, \dots, Q_{r+1}$ satisfy conditions \ref{ends}, \ref{oder-2'}, and \ref{predec}.

We repeat this procedure as long as possible, eventually obtaining a collection of paths $\mathcal{Q} = \{Q_1, \dots, Q_m\}$ satisfying \ref{ends}, \ref{oder-2}, and \ref{predec}. We claim that for any $i < j$, $Q_i$ and $Q_j$ are internally vertex-disjoint. If $i + 2 \leq j$, then \ref{oder-2} implies that $a_i < b_i \leq a_j < b_j$ and thus $Q_i$ and $Q_j$ are internally disjoint since they are both forward paths. In the case $j = i + 1$, suppose for a contradiction that the interiors of $Q_i$ and $Q_j$ intersect at $c \in V(G)$. Then $Q_{i+1}[a_{i+1},c]\cup Q_i[c,b_i]$ is a forward $(a_{i+1},b_i)$-path of length at least 2.
Hence $\left(P\setminus \{a_{i+1}b_i\}\right)\cup Q_{i+1}[a_{i+1},c]\cup Q_i[c,b_i]$ is a forward $(v_1,x_n)$-path of length at least $t$, which is strictly greater than $\ell(P)$, contradicting the maximality of $P$.

\begin{figure}
\centering
\begin{tikzpicture}[
    node/.style={fill, circle, inner sep=1.5pt},
    dot/.style={fill, circle, inner sep=1pt},
    xscale=0.5,
    fixedarc/.style={out=60, in=120, looseness=0.8},
    arc label/.style={below=2pt, font=\footnotesize}
]

\foreach \x in {1,2,...,30} {
    \node[node] at (\x,0) (n\x) {}; 
}

\draw[ultra thick] 
  (1,0) -- (4,0)    
  (5,0) -- (8,0)    
  (9,0) -- (15,0)   
  (17,0) -- (21,0)  
  (22,0) -- (25,0)  
  (26,0) -- (30,0); 
\draw[thin] 
  (4,0) -- (5,0)    
  (8,0) -- (9,0)    
  (15,0) -- (16,0)  
  (16,0) -- (17,0)  
  (21,0) -- (22,0)  
  (25,0) -- (26,0); 

\foreach \i/\s/\t [ 
    evaluate=\s as \startnode using int(\s),
    evaluate=\t as \endnode using int(\t)
] in {
    1/1/5,    
    2/4/9,    
    6/21/26,  
    7/25/30   
} {
    \node[arc label] at (n\startnode) {$a_{\i}$};
    \node[arc label] at (n\endnode) {$b_{\i}$};
}
    \node[arc label] at (15, 0) {$a_4$};
    \node[arc label] at (17, 0) {$b_4$};
    \node[arc label] at (8, 0) {$a_3$};
    \node[arc label] at (16, -0.2) {$b_3=a_5$};
    \node[arc label] at (22, 0) {$b_5$};

\newcommand{\drawarc}[3]{
    \path let \p1=($(#2)-(#1)$) in 
        \pgfextra{
            \pgfmathparse{veclen(\x1,\y1)/28.45274}
            \xdef\arcspan{\pgfmathresult}
        };
    \pgfmathparse{0.8 * sqrt(3/\arcspan)}
    \let\arcflex\pgfmathresult
    
    \pgfmathparse{1.0/(#3+1)}
    \let\step\pgfmathresult
    \pgfmathparse{1.0-\step}
    \let\endpos\pgfmathresult
    
    \draw[black, ultra thick,
          postaction={decorate},
          decoration={
            markings,
            mark=between positions \step and \endpos step (\step) with {
              \node[dot] {};
            }
          }]
         (#1) to[fixedarc, /tikz/looseness=\arcflex] (#2);
}

\begin{scope}[decoration={markings}]
\foreach \s/\t/\n in {
    1/5/2,
    4/9/0,
    8/16/3,
    15/17/1,
    16/22/2,
    21/26/3,
    25/30/3  
} {
    \drawarc{n\s}{n\t}{\n} 
}
\end{scope}

\end{tikzpicture}
\captionsetup{justification=centering}
\caption{An example of the forward path $P$ and the path collection $\mathcal{Q}$ forming a vine. The cycle $C$ is illustrated by the bold line.}\label{figvine}
\end{figure}
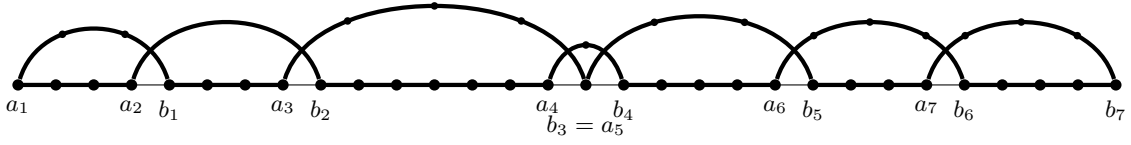

The vine $\mathcal{Q}$ just constructed yields the cycle (cf. \Cref{figvine})
\[C=\big(P\setminus\{a_{j+1}b_j:j\in [m-1]\}\big)\cup\left(\bigcup\limits_{i\in [m]}Q_i[a_i,b_i]\right).\]
In other words, if $m$ is odd, this cycle is precisely 
\[a_1 Q_1 b_1 \overline{P} a_3 Q_3 b_3 \dots a_m Q_m b_m \overline{P} b_{m-1} Q_{m-1} a_{m-1} \dots a_2 \overline{P} a_1, \]
whereas if it is even, the cycle we get is 
\[a_1 Q_1 b_1 \overline{P} a_3 Q_3 b_3 \dots a_{m-1} Q_{m-1} b_{m-1} \overline{P} b_m Q_m a_m \dots a_2 \overline{P} a_1,\]
where we informally write $\overline P$ above to refer to any subpath between two specified vertices on the path $P$.

It remains to verify that $\ell(C)\in [t,2t-2]$.
Since $V(C)\supseteq V(P)$, we have that $\ell(C)\geq t$.
We also have $\ell(Q_i)\leq \ell(P[a_i,b_i])$, as otherwise the forward path $\big(P\setminus P[a_i,b_i]\big)\cup Q_i[a_i,b_i]$ contradicts the maximality of $P$.
Hence, we have $$\ell(C)= \ell(P)+\sum\limits_{i=1}^m\ell(Q_i)-(m-1)\leq \ell(P)+\sum\limits_{i=1}^m\ell(P[a_i,b_i])-(m-1)=2\ell(P)= 2t-2.$$

Let $Q=u_1u_2\cdots u_{\ell+1}$ be a longest forward path in $G$, so that $u_{\ell+1}=x_n$.
Then for every $t\in [2,\ell]$, $Q[u_{\ell+1-t},x_n]$ is a longest forward path with $t+1\geq 3$ vertices starting at the vertex $u_{\ell+1-t}$ in $G$. By the argument above, each of these paths yields a cycle whose length belongs to the interval $[t+1, 2t]$. Thus, $\mathcal{C}_G\cap[t+1,2t]\neq \emptyset$ for every $t\in [2,\ell]$, which implies $\mathcal{C}_G\cap[2^s+1,2^{s+1}]\neq \emptyset$ for every $s\in \big[\lfloor\log \ell\rfloor\big]$. Since the intervals $[2^s+1,2^{s+1}]$ are pairwise disjoint, we obtain $|\mathcal{C}_G|\geq \lfloor\log \ell\rfloor\geq \log (\ell+1)-1$, as desired.
\end{proof} 

Our next goal is to establish a lower bound on $|\mathcal{C}_G|$ under the assumption that $G$ contains no long forward path. Our proof proceeds by defining a suitable partial order on $V(G)$ and then showing that the absence of long forward paths in $G$ implies the absence of long chains in this partial order. Thanks to the following classical theorem, this will allow us to reduce the problem to the case where $G$ has a long antichain. 

\begin{theorem}[Dilworth \cite{dilworth}]\label{dilworth}
In any finite partial order, the maximum size of an antichain is equal to the minimum number of chains required to cover all its elements.
\end{theorem}

Let $(G,\mathcal{X})$ be a $k$-ordered graph and let $V=V(G)$. 
For $u,v\in V$, let $u\preceq v$ if there exists a forward $(u,v)$-path in $G$.
It is easy to see that \((V, \preceq)\) is a partial order.
We call $(V,\preceq)$ the partial order \textit{generated by} $\mathcal{X}$.
We also write $u\prec v$ when $u\preceq v$ and $u\neq v$.
Observe that if $v_1\prec v_2\prec \cdots\prec v_\ell$ is a chain under the partial order $(V,\preceq)$, then there exists a forward path $P$ such that $v_1,v_2,\cdots, v_\ell$ occur sequentially along $P$.
Hence if every forward path in $G$ has length at most $\ell$, then every chain under $(V,\preceq)$ contains at most $\ell+1$ elements. If so, by Theorem~\ref{dilworth}, $(V, \preceq)$ contains an antichain on at least $n/(\ell + 1)$ elements.

The next lemma shows that, given any antichain $L$, one can find two trees whose leaf set is precisely $L$ and which have no other vertices in common.
A subtree $T$ of $G$ rooted at $u$ is called \textit{forward-directed} (resp. \textit{backward-directed}) if
\begin{itemize}
    \item[(1)] for any subpath $P=u_1u_2\cdots u_t$ of $T$ with $u_1=u$ and $u_t\in L(T)$, $u_i\preceq u_{i+1}$ (resp. $u_{i+1}\preceq u_i$) for every $i\in [t-1]$; and
    \item[(2)] either $d_T(u)\geq 2$ or $T$ consists of a single vertex.
\end{itemize}
Hence the root of a forward-directed (resp. backward-directed) tree $T$ is its minimum (maximum) vertex under $\preceq$.

\begin{lemma}\label{lem:2tree}
     Let $(G,\mathcal{X})$ be a $k$-ordered graph for some $k \geq 3$ and let $(V,\preceq)$ be the partial order generated by $\mathcal{X}$.
     Then for any given antichain $L=\{v_1,v_2,\cdots,v_m\}$ under $\preceq$, there exist a forward-directed subtree $S$ and a backward-directed subtree $T$ of $G$ satisfying $L(S)=L(T)=L$.
\end{lemma}

\begin{proof}
    
We will just prove the existence of a forward-directed tree $S$ such that $L(S)=L$, since, as will be clear by the end of the proof, the existence of the required backward-directed tree follows by symmetry.

$S$ is constructed through the following procedure. At the start, we let $S_1 = \{v_1\}$, which we view as a one-vertex tree rooted at $v_1$. Now, suppose that we have already constructed a forward-directed tree $S_i$ for some integer $i\in [m-1]$, and that $L(S_i)=\{v_1,\cdots, v_i\}$. Let $u_i$ be the root of $S_i$.
Then $v_{i+1}\preceq u_i$ cannot hold, as otherwise $v_{i+1}\preceq u_i\preceq v_1$, contradicting the fact that $L$ forms an antichain under $\preceq$.

We will now show how to extend $S_i$ to a larger forward-directed tree $S_{i+1}$ with $L(S_{i+1}) = \{v_1, \dots, v_{i+1}\}$. We split into two cases. 

\medskip

\ninn {\bf Case 1:} $u_i\prec v_{i+1}$.

Note that this cannot happen when $i=1$, hence we may assume that $|L(S_i)|\geq 2$ and $d_{S_{i}}(u_i)\geq 2$.
Let $v\in V(S_{i})$ be a maximal vertex under $\preceq$ such that $v\preceq v_{i+1}$.
Select an arbitrary forward $(v,v_{i+1})$ path $P$ in $G$, which implies $V(P)\cap V(S_{i})=\{v\}$ by maximality of $v$.
Let $S_{i+1}=S_i\cup P$, and observe that $d_{S_{i+1}}(u_i)\geq d_{S_{i}}(u_i)\geq 2$.
Thus, $S_{i+1}$ is a forward-directed tree rooted at $u_i$ such that $L(S_{i+1})=\{v_1,\cdots, v_{i+1}\}$.

\medskip

\ninn {\bf Case 2:} $u_i\nprec v_{i+1}$.

Let $w\in V(G)$ be a maximal vertex under $\preceq$ such that $w\preceq v_{i+1}$ and $w\preceq u_i$. Indeed, such a vertex $w$ exists since $x_1 \preceq u_i$ and $x_1 \preceq v_{i+1}$ (recall that each vertex $x \in V(G) \setminus \{x_1\}$ satisfies $d^-(x) \geq 1$, and thus $x_1 \preceq x$). Select an arbitrary forward $(w,u_i)$ path $P$ and an arbitrary forward $(w,v_{i+1})$ path $Q$.
Then, our choice of $w$ and the fact that $u_i \nprec v_{i+1}$ imply that $V(P)\cap V(S_i)=\{u_i\}$, $V(Q)\cap V(S_i)=\emptyset$, and $V(P)\cap V(Q)=\{w\}$.
Letting $S_{i+1}=S_i\cup P\cup Q$, we have $d_{S_{i+1}}(w)\geq 2$.
Thus, $S_{i+1}$ is a forward-directed tree rooted at $w$ such that $L(S_{i+1})=\{v_1,\cdots, v_{i+1}\}$.

\medskip

The algorithm terminates with a forward-directed tree $S_m$ with $L(S_m)=L$, as required. It can be easily checked that the same argument yields a backward-directed tree $T$ such that $L(T) = L$. The only part of the argument that does not follow directly from the symmetry of the partial order is in Case 2, where we instead use the fact that $x \preceq x_n$ for each $x \in V(G) \setminus \{x_n\}$ since $d^+(x)\geq 1$. 
\end{proof} 

We call a tree $T$ rooted at $u$ \emph{fair} if, for some $q \geq 1$, each leaf $x \in L(T)$ satisfies $d(u,x) = q$. The following lemma shows that by reducing the size of the antichain $L$ by at most a constant factor, we can essentially assume that the forward-directed and backward-directed subtrees guaranteed by Lemma~\ref{lem:2tree} are both fair.

\begin{lemma}\label{lem:perfect 2tree}
    Let $k \geq 3, c \geq 1$. Let $(G,\mathcal{X})$ be a $k$-ordered graph with no forward path of length $c$, and let $(V, \preceq)$ be the partial order generated by $\mathcal{X}$. Then $G$ contains an antichain $L_0 \subseteq V$, a fair forward-directed tree $S_0$, and a fair backward-directed tree $T_0$, satisfying $L(S_0) = L(T_0) = L_0$. Moreover, $|L_0|\geq\frac{|V|}{c^3}$.
\end{lemma}

\begin{proof} First, observe that if $|V| \leq c^3$, then the statement can be seen to be trivially true by choosing $L_0 = \{v\}$ where $v \in V$ is arbitrary, and letting $S_0$ and $T_0$ be one-vertex trees with vertex set $\{v\}$. With this choice, $|L_0| = 1 \geq |V|/c^3$.

From now on, we will assume that $|V| > c^3$. By our assumption on the length of forward paths in $(G, \mathcal{X})$, every chain under $\preceq$ contains at most $c$ elements. By Theorem~\ref{dilworth}, this implies that there is an antichain $L$ satisfying $|L|\geq \frac{|V|}{c}$. 

By Lemma~\ref{lem:2tree}, there exist a forward-directed tree $S$ rooted at $u$ and a backward-directed tree $T$ rooted at $v$ with $L(S)=L(T)=L$. Observe that the path in $S$ connecting $u$ to any given $w \in L$ is a forward path, and thus of length at most $c$. So, there is a subset $L_1\subseteq L$ with $|L_1|\geq \frac{|L|}{c} \geq \frac{|V|}{c^2} \geq 2$ such that any two leaves in $L_1$ are at the same distance from $u$ in $S$. Let $S_1$ and $T_1$ be the unique subtrees of $S$ and $T$ such that $L(S_1)=L(T_1)=L_1$. Let $u'\in V(S_1)$ be the minimum vertex under $\prec$. Then $d_{S_1}(u') \geq 2$ and $S_1$ is a forward-directed tree rooted at $u'$. Analogously, by choosing $v'\in V(T_1)$ to be the maximum vertex under $\prec$, we get that $T_1$ is a backward-directed tree rooted at $v'$. Moreover, $S_1$ is fair. 

We now apply a similar procedure to $T_1$. Again, there must be a subset $L_0 \subseteq L_1$ with $|L_0|\geq \frac{|L_1|}{c} \geq \frac{|V|}{c^3} >1$ such that any two leaves in $L_0$ are at the same distance from $v'$ in $T_1$. Let $S_0$ and $T_0$ be the unique subtrees of $S_1$ and $T_1$ respectively, such that $L(S_0)=L(T_0)=L_0$. By the same argument as before, $S_0$ is a forward-directed tree and $T_0$ is a backward-directed tree. Moreover, both $S_0$ and $T_0$ are fair, as required. 
\end{proof} 

Next, we obtain a lower bound on $|\mathcal{C}_G|$ using the structure from Lemma~\ref{lem:perfect 2tree}. It will be sufficient for our purposes to consider cycles of a special kind. We call a cycle $C$ \emph{good} if $C$ is the union of two internally-disjoint forward $(u,v)$-paths for two vertices $u,v\in V$. Denote the set of all lengths of good cycles in $G$ by $\mathcal{C}_1(G)$.

\begin{lemma}\label{lem:C1} Let $k\geq 3, \Delta \geq 2$, and let $(G,\mathcal{X})$ be a $k$-ordered graph. Suppose that $S$ is a fair forward-directed tree in $G$, and that $T$ is a fair backward-directed tree in $G$, such that $L(S) = L(T) =L$ where $|L| \geq 2$. Further assume that $\Delta(S) \leq \Delta$.
    Then $|\mathcal{C}_1(S\cup T)|\geq \frac{\log |L|}{\log \Delta}$.
\end{lemma}

\begin{proof}

Let $(V,\preceq)$ be the partial order generated by $\mathcal{X}$.
We prove the lemma by induction on $|L|$.
As a base case, let $|L|\in [2,\Delta]$. Pick any two leaves in $L$, and note that they are connected by a path $P_1$ in $S$ and a path $P_2$ in $T$. Then, $P_1 \cup P_2$ is a good cycle, so that $|\mathcal{C}_1(S\cup T)|\geq 1 \geq \frac{\log |L|}{\log \Delta}$. 

Assume that the lemma holds for $|L|\leq t-1$ and consider the case $|L|=t>\Delta$. Let $u$ and $v$ be the roots of $S$ and $T$ respectively. Let $r_1, \dots, r_{\Delta'}$ ($\Delta' \leq\Delta$) be the neighbours of $u$ in $S$. Observe that for some $i \in [\Delta']$ the subtree $S'$ rooted at $r_i$ obtained by deleting $u$ from $S$ contains at least $|L|/\Delta \geq 2$ leaves distinct from $r_i$. Let $u_0$ be maximal under $\preceq$ such that $u_0$ is contained in every path from $r_i$ to $L(S')$ in $S$. Let $S_0$ be the unique subtree of $S'$ whose leaf set is precisely $L_0 := L(S')$, then $S_0$ is a fair forward-directed tree rooted at $u_0$ with $|L_0|\geq |L|/\Delta$. Let $T_0\subseteq T$ be the unique backward-directed subtree of $T$ with $L(T_0)=L_0$. Letting $v_0$ be the minimal vertex under $\preceq$ that is on every path from $v$ to $L_0$ in $T$, we view $T_0$ as rooted at $v_0$, so that $T_0$ is also fair.

By the inductive hypothesis applied to $L_0$, $S_0$ and $T_0$, we have \[|\mathcal{C}_1(S_0\cup T_0)|\geq \frac{\log|L_0|}{\log \Delta} \geq \frac{\log|L|}{\log \Delta}-1.\]

Let $d_S, d_T \geq 1$ satisfy $d_S(u, w) = d_S$ and $d_T(v, w) = d_T$ for each $w \in L$. Any cycle contained in $S_0 \cup T_0$ is of length at most $2 d_{S} + 2 d_{T}- 2$. Therefore, to complete the proof it suffices to show that there exists a good cycle in $S \cup T$ containing $u$ and $v$, which must have length precisely $2d_S + 2d_T$.

Suppose otherwise. By our choice of $L_0$, every subpath of $S$ connecting $L_0$ and $L \setminus L_0$ must contain $u$. So, we may assume that every subpath in $T$ connecting $L_0$ and $L \setminus L_0$ avoids $v$. Let $r'_1, \dots, r'_{\Delta''}$ ($\Delta'' \leq \Delta$) be the neighbours of $v$ in $T$, and let $T_{i}$ ($i \in [\Delta'']$) be the subtree of $T$ containing $r'_i$ after deleting $v$. If there are distinct $i, j \in [\Delta'']$ such that $L_0 \cap V(T_i)$ and $(L \setminus L_0) \cap V(T_j)$ are non-empty, then we obtain a path from $L$ to $L\setminus L_0$ in $T$ containing $v$ (passing through $r'_i$ and $r'_j$), giving a contradiction. Thus, there is some $i \in [\Delta'']$ such that $L = L_0 \cup (L \setminus L_0) \subseteq V(T_i)$, which is only possible if $d_T(v) = 1$, contradicting the fact that $T$ is a backward-directed tree.

Hence, there exists a good cycle in $S \cup T$ containing $u$ and $v$. This cycle is necessarily of length $2d_S + 2d_T$, and thus \[|\mathcal{C}_1(S\cup T)|\geq |\mathcal{C}_1(S_0\cup T_0)|+1\geq \frac{\log|L|}{\log \Delta},\] which completes the proof. \end{proof}

Finally, we are ready to complete the proof of Theorem~\ref{thm:ord->uppc(G)}.

\begin{proof}[ Proof of Theorem~\ref{thm:ord->uppc(G)}.] 
Suppose the maximum length of a forward path in $G$ is $c-1$.
By Lemma~\ref{lem:vine}, $|\mathcal{C}_G|\geq \log c-1$. 
Hence if $n<2c^3$, we have $|\mathcal{C}_G|\geq \frac{\log n-4}{3}>\frac{\log n}{3+\log k}-2$. 
Consider the case $n\geq 2c^3$. Applying Lemma~\ref{lem:perfect 2tree}, we obtain $L\subseteq V$, a fair forward-directed subtree $S$ and a fair backward-directed subtree $T$ of $G$, satisfying $L(S)=L(T)=L$ and $|L|\geq \frac{n}{c^3} \geq 2$.
Since $(G,\mathcal{X})$ is $k$-ordered,
$S$ has maximum degree at most $k$.
From Lemma~\ref{lem:C1}, it follows that $$|\mathcal{C}_G|\geq\big|\mathcal{C}_1(S\cup T)\big|\geq \frac{\log|L|}{\log k}\geq \frac{\log n-3\log c}{\log k}.$$ 
Now we complete the proof by deducing that
\begin{align*}
\big|\mathcal{C}_G\big|
&\geq \min_{c>0}\max\left\{\log c-1,\frac{\log n-3\log c}{\log k}\right\}=\frac{\log n-3}{3+\log k},
\end{align*}
where $\max\left\{\log c-1,\frac{\log n-3\log c}{\log k}\right\}$ achieves its minimum when $\log c=\frac{\log n+\log k}{3+\log k}$.
\end{proof}

Theorem~\ref{thm:d-3-cCL} promptly follows by setting $k = 3$ and combining Lemma~\ref{lem:orien exist} with Theorem~\ref{thm:ord->uppc(G)}.

\section{Conclusion and open problems}\label{sec:conclud}
In this paper, we answered several questions of Narins, Pokrovskiy and Szab\'o \cite{NPSz} on lengths of cycles in degree-critical graphs and leaf-to-leaf paths in trees. We have proven \Cref{conj:anylengths} and disproven \Cref{conj:smalllengths}, but several questions still remain. The most obvious one would be to improve the leading coefficient of the bound we prove in \Cref{thm:d-3-cCL} and completely settle \Cref{conj:manycycles}.

Another interesting question is to determine `how far' \Cref{conj:smalllengths} is from being true, i.e. find the value of the best possible constant $c$ in \Cref{thm:upper-bound-lengths}.

\begin{problem}
    Determine the supremum $c^*$ over all $c\in[0, 1]$ for which the following holds: for all $N$ and all sufficiently large even $n$ (as a function of $N$ and $c$), every $n$-vertex 1--3 tree contains leaf-to-leaf paths of $\Omega(N^{c})$ distinct lengths between $0$ and $N$.
\end{problem}

We do not have a guess for what the true value of $c^*$ should be. \Cref{smalllengths} shows that $c^*\geq 2/3$. In the proof of \Cref{smalllengths}, however, we could only obtain leaf-to-leaf paths which are all witnessed by the same leaf, and \Cref{lots} shows that our bound in this setting is essentially best possible. It is natural to attempt and improve this lower bound on $c^*$ by sharpening the bound in \Cref{additive}, i.e. improving on the lower bound $c'\geq 2/3$ in the setting below.

\begin{problem}\label{conj:additive} Determine the supremum $c'$ over all $c\in [0, 1]$ for which the following holds: for all sufficiently large $n$ and all sequences $(a_i)_{i=1}^n$ of non-negative integers such that $a_i \leq n^{c}$, we have
$$|\{a_i + a_j + (j - i) : 1 \leq i < j \leq n\}| =\Omega\left(n^{c}\right).$$  
\end{problem}

On the other hand, \Cref{thm:upper-bound-lengths} shows that $c^*\leq\left(2-\frac{\log 10}{\log 13} \right)^{-1}\approx 0.9073$, and a straightforward application of the Ruzsa triangle inequality shows that our proof method cannot improve this beyond $0.75$ (more specifically, it is proven in \cite{ruzsa} that one has $|U+V|\geq |U-V|^{2/3}$ for any $U, V\subseteq \mathbb{Z}$, so we cannot take $\beta<2/3$ in \Cref{prop:additive-to-tree}).

As a related problem, it would be interesting to determine the optimal value of $\beta$ that one could take in \Cref{prop:additive-to-tree}. We remark that even the more basic question of determining how small $A+B$ can be relative to $A-B$ for $A, B\subseteq \mathbb{N}$ seems to be wide open -- the best bound we are aware of is the construction of Cutler, Pebody, and Sarkar \cite{sumset} which gives $|A+A|\leq |A-A|^{0.868}.$

Lastly, let us mention one more problem stated in \cite{NPSz}.

\begin{problem}[\!\!{\cite[Problem 6.1]{NPSz}}]\label{pb:allcycles}
    Is there a function $C(n)$ tending to infinity such that every degree 3-critical graph on $n$ vertices contains cycles of all lengths $4, 6, 8, . . . , 2C(n)$?
\end{problem}

The tools used in the present paper seem insufficient to be able to answer this, and we do not speculate on what the answer might be.

\section*{Acknowledgements}
We thank the anonymous referees for their careful reading of this paper and their valuable comments.
We would also like to thank Jozef Skokan for a careful reading of a preliminary version of this manuscript.

\bibliographystyle{plain}
\bibliography{bibliography}

@article{NPSz,
    author={Narins, L. and Pokrovskiy, A. and Szab\'o, T.},
    title={Graphs without proper subgraphs of minimum degree 3 and short cycles},
    journal={Combinatorica},
    year={2017},
    pages={495-519},
    volume={37}
}

@book{bollobas,
    author={Bollob\'as, B.},
    title={Modern Graph Theory},
    year={1998},
    publisher={Springer, New York},
    series={Graduate Texts in Mathematics},
    volume={184}
}

@incollection {GLhelly,
    AUTHOR = {Gy\'arf\'as, A. and Lehel, J.},
     TITLE = {A {H}elly-type problem in trees},
 BOOKTITLE = {Combinatorial theory and its applications, {I}-{III} ({P}roceedings of the
              {C}olloquium held at {B}alatonf\"ured, 1969)},
    SERIES = {Colloquia Mathematica Societatis J\'anos Bolyai},
    VOLUME = {4},
     PAGES = {571--584},
 PUBLISHER = {North-Holland, Amsterdam-London},
      YEAR = {1970},
   MRCLASS = {52A35 (05C05)},
  MRNUMBER = {298550},
MRREVIEWER = {Branko\ Grunbaum},
}

@article{Bucić_Gishboliner_Sudakov_2022,
    title={{Cycles of many lengths in Hamiltonian graphs}},
    volume={10},
    DOI={10.1017/fms.2022.42},
    journal={Forum of Mathematics, Sigma},
    author={Buci{\'c}, M. and Gishboliner, L. and Sudakov, B.},
    year={2022},
    pages={e70}
}

@article{nashwill,
    author = {Nash-Williams, C. St. J. A.},
    title = {Decomposition of Finite Graphs Into Forests},
    journal = {Journal of the London Mathematical Society},
    volume = {s1-39},
    number = {1},
    pages = {12-12},
    doi = {https://doi.org/10.1112/jlms/s1-39.1.12},
    url = {https://londmathsoc.onlinelibrary.wiley.com/doi/abs/10.1112/jlms/s1-39.1.12},
    eprint = {https://londmathsoc.onlinelibrary.wiley.com/doi/pdf/10.1112/jlms/s1-39.1.12},
    year = {1964}
}

@article{helly, 
    title={Three results for trees, using mathematical induction},
    author={Horn, W. A.},
    journal={Journal of Research of the National Bureau of Standards},
    volume={76B},
    pages={39--43},
    year={1972}
}

@article{erdos1988cycles,
  title={Cycles in graphs without proper subgraphs of minimum degree 3},
  author={Erd\H{o}s, P. and Faudree, R. J. and Gy{\'a}rf{\'a}s, A. and Schelp, R. H.},
  journal={Ars Combinatorica},
  volume={25},
  pages={195--201},
  year={1988}
}

@article{erdos1935combinatorial,
  title={A combinatorial problem in geometry},
  author={Erd\H{o}s, P. and Szekeres, G.},
  journal={Compositio Mathematica},
  volume={2},
  pages={463--470},
  year={1935}
}

@article{bondy1971pancyclic,
  title={{Pancyclic graphs I}},
  author={Bondy, J. A.},
  journal={Journal of Combinatorial Theory, Series B},
  volume={11},
  number={1},
  pages={80--84},
  year={1971}
}

@article{sudakov2008cycle,
  title={Cycle lengths in sparse graphs},
  author={Sudakov, B. and Verstra{\"e}te, J.},
  journal={Combinatorica},
  volume={28},
  number={3},
  pages={357--372},
  year={2008},
  publisher={Springer}
}

@article{erdos1993some,
  title={Some of my favorite solved and unsolved problems in graph theory},
  author={Erd\H{o}s, P.},
  journal={{Quaestiones Mathematicae}},
  volume={16},
  number={3},
  pages={333--350},
  year={1993},
  publisher={Taylor \& Francis}
}

@article{gyarfas1984distribution,
  title={On the distribution of cycle lengths in graphs},
  author={Gy{\'a}rf{\'a}s, A. and Koml{\'o}s, J. and Szemer{\'e}di, E.},
  journal={{Journal of Graph Theory}},
  volume={8},
  number={4},
  pages={441--462},
  year={1984},
  publisher={Wiley Online Library}
}

@article{letzter2023pancyclicity,
  title={Pancyclicity of highly connected graphs},
  author={Letzter, S.},
  journal={arXiv preprint arXiv:2306.12579},
  year={2023}
}

@article{draganic2024pancyclicity,
  title={{Pancyclicity of Hamiltonian graphs}},
  author={Dragani{\'c}, N. and Correia, D. M. and Sudakov, B.},
  journal={{Journal of the European Mathematical Society}},
  year={2024}
}

@article{bauer1990hamiltonian,
  title={{Hamiltonian degree conditions which imply a graph is pancyclic}},
  author={Bauer, D. and Schmeichel, E.},
  journal={Journal of Combinatorial Theory, Series B},
  volume={48},
  number={1},
  pages={111--116},
  year={1990},
  publisher={Elsevier}
}

@incollection{bollobas1989long,
  title={Long cycles in graphs with no subgraphs of minimal degree 3},
  author={Bollob{\'a}s, B. and Brightwell, G.},
  booktitle={Annals of Discrete Mathematics},
  volume={43},
  pages={47--53},
  year={1989},
  publisher={Elsevier}
}

@article{bondy1981relative,
  title={Relative lengths of paths and cycles in 3-connected graphs},
  author={Bondy, J. A. and Locke, S. C.},
  journal={Discrete Mathematics},
  volume={33},
  number={2},
  pages={111--122},
  year={1981},
  publisher={Elsevier}
}

@article{dilworth,
 ISSN = {0003486X, 19398980},
 URL = {http://www.jstor.org/stable/1969503},
 author = {R. P. Dilworth},
 journal = {Annals of Mathematics},
 number = {1},
 pages = {161--166},
 publisher = {[Annals of Mathematics, Trustees of Princeton University on Behalf of the Annals of Mathematics, Mathematics Department, Princeton University]},
 title = {A Decomposition Theorem for Partially Ordered Sets},
 urldate = {2025-02-27},
 volume = {51},
 year = {1950}
}

@incollection{ruzsa,
    author={Ruzsa, I.},
    title={Sums of finite sets},
    booktitle={Number Theory: New York Seminar},
    editor={D.V. Chudnovsky and G.V. Chudnovsky and M.B. Nathanson},
    publisher={Springer-Verlag},
    year={1996}
}

@article{sumset,
    author={Cutler, J. and Pebody, L. and Sarkar, A.},
    title={Sums, {D}ifferences and {D}ilates},
    year={2024},
    journal={arXiv preprint arXiv:2402.18297
        
        
        
        }
}

@article{sauermann,
    author={Sauermann, L.},
    title={A proof of a conjecture of {E}rd{\H{o}}s, {F}audree, {R}ousseau and {S}chelp on subgraphs of minimum degree $k$},
    journal={Journal of Combinatorial Theory, Series B},
    volume={134},
    year={2019},
    pages={36-75},

}

@article{Erd91,
    author={Erd\H{o}s, P.},
    title={Problems and results in combinatorial analysis and combinatorial number theory},
    journal={Graph theory, combinatorics, and applications},
    volume={Vol. 1 (Kalamazoo, MI, 1988)},
    year={1991},
    pages={397-406},

}

\end{document}